\documentclass[reqno,final]{amsart}

\usepackage{natbib}   % Nature-like bibliography
\usepackage{fancyhdr} % Headers
\usepackage{color}    % Color definition
\usepackage{hyperref} % Internal and external links
\usepackage{graphicx} % Graphics inclusion
\usepackage{epstopdf} % EPS extension to PDF extension
\usepackage{srcltx}   % Forward dvi

\usepackage{amssymb}
\usepackage{caption}
\usepackage{mathrsfs}
\usepackage{enumerate}
\usepackage{amsfonts}

\definecolor{aleacolor}{rgb}{0.16,0.59,0.78}

\hypersetup{
breaklinks,
colorlinks=true,
linkcolor=aleacolor,
urlcolor=aleacolor,
citecolor=aleacolor}

\renewcommand{\headheight}{11pt}
\renewcommand{\cite}{\citet}

\theoremstyle{plain}
\newtheorem{theorem}{Theorem}[section]

\newtheorem{lemma}[theorem]{Lemma}

\theoremstyle{definition}
\newtheorem{definition}[theorem]{Definition}
\theoremstyle{remark}

\makeatletter \@addtoreset{equation}{section} \makeatother
\renewcommand\theequation{\thesection.\arabic{equation}}
\renewcommand\thefigure{\thesection.\arabic{figure}}
\renewcommand\thetable{\thesection.\arabic{table}}

\elogo{\parbox[c]{78pt}{\includegraphics[width=78pt]{logo.pdf}}} % 3 centimeters = 85.0393701 PostScript points

\newcommand{\N}{\mathbb{N}}
\newcommand{\Z}{\mathbb{Z}}
\newcommand{\Q}{\mathbb{Q}}
\newcommand{\R}{\mathbb{R}}
\newcommand{\PP}{\mathbb{P}}
\newcommand{\F}{\mathcal{F}}
\newcommand{\nvert}[0]{~\vert~}
\newcommand{\prob}[2]{\mathbb{P}_{#1}\hspace{-0.3mm}\left( #2 \right)}
\newcommand{\esp}[2]{\mathbb{E}_{#1}\hspace{-0.3mm}\left[ #2 \right]}
\newcommand{\var}[2]{\mathbb{V}_{#1}\hspace{-0.3mm}\left( #2 \right)}
\newcommand{\probw}[2]{\mathscr{P}_{#1}\hspace{-0.3mm}\left( #2 \right)}
\newcommand{\espw}[2]{\mathscr{E}_{#1}\hspace{-0.3mm}\left[ #2 \right]}
\newcommand{\overbar}[1]{\mkern 1.5mu\overline{\mkern-1.5mu#1\mkern-1.5mu}\mkern 1.5mu}

\begin{document}

\title[Scale-inhomogeneous Gaussian free field]{Extremes of the two-dimensional Gaussian free field with scale-dependent variance\\ {\normalsize See \href{http://alea.impa.br/english/index_v13.htm}{ALEA} for the official version}}

\author[L.-P. Arguin]{Louis-Pierre Arguin$^1$}
\address{Baruch College and Graduate Center (CUNY), Department of Mathematics \newline New York, NY 10010, USA.}
\email{louis-pierre.arguin@baruch.cuny.edu}
\urladdr{\url{https://arguin.commons.gc.cuny.edu}}
\thanks{$^1$L.-P. Arguin is supported by the NSF grant DMS 1513441, the PSC-CUNY Research Award~68784-00~46, a Eugene M. Lang Junior Faculty Research Fellowship, and partially by a NSERC Discovery grant and a FQRNT {\it Nouveaux chercheurs} grant.}

\author[F. Ouimet]{Fr\'ed\'eric Ouimet$^2$}
\address{Universit\'e de Montr\'eal, D\'epartement de Math\'ematiques et de Statistique \newline Montr\'eal, QC H3T 1J4, Canada.}
\email{ouimetfr@dms.umontreal.ca}
\urladdr{\url{https://sites.google.com/site/fouimet26}}
\thanks{$^2$F. Ouimet is supported by a NSERC Doctoral Program Alexander Graham Bell scholarship and partially by a NSERC Master Program Alexander Graham Bell scholarship.}

\subjclass[2010]{60G70, 82B44}
\keywords{extreme value theory, Gaussian free field, branching random walk}

    \begin{abstract}
        In this paper, we study a random field constructed from the two-dimensional Gaussian free field (GFF) by modifying the variance along the scales in the neighborhood of each point.
        The construction can be seen as a local martingale transform and is akin to the time-inhomogeneous branching random walk.
        In the case where the variance takes finitely many values, we compute the first order of the maximum and the log-number of high points.
        These quantities were obtained by \citet{MR1880237} and \cite{MR2243875} when the variance is constant on all scales.
        The proof relies on a truncated second moment method proposed by \cite{MR3380419}, which streamlines the proof of the previous results.
        We also discuss possible extensions of the construction to the continuous GFF.
    \end{abstract}

    \maketitle

\section{Introduction}
    \subsection{The model}\label{sec:model}

    Let $(W_k)_{k\geq 0}$ be a simple random walk starting at $u\in \Z^2$ with law $\mathscr{P}_u$. For every finite box $B \subseteq \Z^2$, the Gaussian free field (GFF) on $B$ is a centered Gaussian field $\phi \circeq \{\phi_v\}_{v\in B}$ with covariance matrix
    \begin{equation}\label{eqn: G}
        G_B(u,v) \circeq \frac{\pi}{2} \cdot \espw{u}{\sum_{k=0}^{\tau_{\partial B} - 1} 1_{\{W_k = v\}}}, \ \ \ u,v\in B,
    \end{equation}
    where $\tau_{\partial B}$ is the first hitting time of $(W_k)_{k\geq 0}$ on the \textit{boundary of $B$},
    \begin{equation*}
        \partial B \circeq \{v\in B \nvert \exists z\not\in B ~~\text{such that } \|v - z\|_2 = 1\},
    \end{equation*}
    and $\|\cdot\|_2$ denotes the Euclidean distance in $\Z^2$.
    With this definition, $B$ contains its boundary.
    We let $B^o \circeq B \backslash \partial B$.
    By convention, summations are zero when there are no indices, so $\phi$ is identically zero on $\partial B$. This is the {\it Dirichlet boundary condition}.
    The constant $\pi/2$ in \eqref{eqn: G} is a convenient normalization for the variance.

    In this paper, we consider a family of Gaussian fields constructed from the GFF $\{\phi_v\}_{v\in V_N}$ on the square box $V_N \circeq \{0,1,...,N\}^2$.
    These Gaussian fields are the analogues, in the context of the GFF, of the time-inhomogeneous branching random walks studied in \citet{MR2070335,MR2968674,MR3164771,arXiv:1509.08172}.
    We study the maxima and the number of high points of this family of Gaussian fields as $N \rightarrow \infty$.

    The construction is very natural for any Gaussian field on a metric space and bears strong similarities with martingale transforms.
    It is based on the modification of the variance in neighborhoods around every point along different mesoscopic scales.
    More precisely, for $\lambda \in (0,1)$ and $v=(v_1,v_2)\in V_N$, consider the closed neighborhood $[v]_\lambda$ in $V_N$ consisting of the square box of width $N^{1-\lambda}$ centered at $v$ that has been cut off by the boundary of $V_N$ :
    \begin{equation*}
        [v]_{\lambda} \circeq \left(\left[v_1 - \frac{1}{2}N^{1-\lambda},v_1 + \frac{1}{2}N^{1-\lambda}\right] \times \left[v_2 - \frac{1}{2}N^{1-\lambda},v_2 + \frac{1}{2}N^{1-\lambda}\right]\right) \bigcap V_N\ .
    \end{equation*}
    By convention, we define $[v]_0 \circeq V_N$ and $[v]_1 \circeq \{v\}$.
    We stress that square boxes are not essential to the construction; any neighborhood centered at $v$ containing points at distance roughly $N^{1-\lambda}$ would do.
    Let $\F_{\partial [v]_{\lambda} \cup [v]_{\lambda}^c} \circeq \sigma(\{\phi_v, v\notin [v]_{\lambda}^o\})$ be the $\sigma$-algebra generated by the variables on the boundary of the box $[v]_{\lambda}$ and those outside of it.
    Since the neighborhoods are shrinking with $\lambda$, for any $v\in V_N$, the collection $\mathbb{F}_v \circeq \{\F_{\partial [v]_{\lambda} \cup [v]_{\lambda}^c}\}_{\lambda\in [0,1]}$ is a filtration.
    In particular, if we let
    \begin{equation*}
        \phi_v(\lambda) \circeq \esp{}{\phi_v \nvert \F_{\partial [v]_{\lambda} \cup [v]_{\lambda}^c}},
    \end{equation*}
    then
    \begin{equation*}
        \text{for every $v\in V_N$, $(\phi_v(\lambda))_{\lambda\in[0,1]}$ is a $\mathbb{F}_v$-martingale.}
    \end{equation*}
    It is also a Gaussian field, therefore disjoint increments of the form $\phi_v(\lambda') - \phi_v(\lambda)$ are independent.
    These observations motivate the definition of {\it scale-inhomogeneous Gaussian free field}, which can be seen as a martingale-transform of $(\phi_v(\lambda))_{\lambda\in[0,1]}$ applied simultaneously for every $v\in V_N$.

    Fix $M\in \N$ and consider the parameters
    \begin{equation*}
        \begin{aligned}
            \quad \quad \quad \quad \boldsymbol{\sigma} &\circeq (\sigma_1,\sigma_2, ..., \sigma_M)\in (0,\infty)^M, ~~~~~~~~~~~~~~~~~~~~~~~~~~~& \text{(variance parameters)} \\
            \boldsymbol{\lambda} &\circeq (\lambda_1,\lambda_2, ..., \lambda_M)\in (0,1]^M, ~\quad \quad \quad \quad \quad \quad \quad \quad & \text{(scale parameters)}
        \end{aligned}
    \end{equation*}
    where $0 \circeq \lambda_0 < \lambda_1 < ... < \lambda_M \circeq 1$.
    The parameters $(\boldsymbol{\sigma}, \boldsymbol{\lambda})$ can be encoded simultaneously in the left-continuous step function
    \begin{equation*}
        \sigma(s) \circeq \sigma_1 1_{\{0\}}(s) + \sum_{i=1}^M \sigma_i 1_{(\lambda_{i-1},\lambda_i]}(s), \ \ \ s\in [0,1]\ .
    \end{equation*}
    We write $\nabla_i$ for the difference operator with respect to the index $i$.
    When the index variable is obvious, we omit the subscript.
    For example,
    \begin{equation*}
        \nabla \phi_v(\lambda_i) \circeq \phi_v(\lambda_i) - \phi_v(\lambda_{i-1})\ .
    \end{equation*}

    \begin{definition}[Scale-inhomogeneous Gaussian free field]\label{def:IGFF}
        \hspace{-0.3mm}Let \hspace{-0.3mm}$\phi \circeq \{\phi_v\}_{v\in V_N}$ \hspace{-0.3mm}be \hspace{-0.3mm}the GFF on $V_N$.
        The $(\boldsymbol{\sigma},\boldsymbol{\lambda})$-GFF on $V_N$ is a Gaussian field $\psi \circeq \{\psi_v\}_{v\in V_N}$ defined by
        \vspace{-2mm}
        \begin{equation}
            \label{eqn: psi v}
            \psi_v \circeq \sum_{i=1}^M \sigma_i \nabla \phi_v(\lambda_i)= \sum_{i=1}^M \sigma_i \big(\phi_v(\lambda_i)-\phi_v(\lambda_{i-1})\big)\ .
        \end{equation}
        Similarly to the GFF, we define
        \begin{equation*}
            \psi_v(\lambda) \circeq \esp{}{\psi_v \nvert \F_{\partial [v]_{\lambda} \cup [v]_{\lambda}^c}}.
        \end{equation*}
    \end{definition}
    The field with two variances ($M=2$) was presented in \citet{MR3354619}, where it was used to prove Poisson-Dirichlet statistics of the Gibbs measure in the homogeneous case ($M=1$).

    \subsection{Main results}

    The main results of this paper are the derivation of the first order of the maximum and the log-number of high points for the scale-inhomogeneous Gaussian free field of Definition \ref{def:IGFF}.
    The methods of proof are general and directly applicable to time-inhomogeneous branching random walks and to other log-correlated Gaussian fields.

    First, we need to introduce some notations.
    For any positive measurable function $f:[0,1]\to\R$, define the integral operators
    \begin{equation*}
        \mathcal{J}_f(s) \circeq \int_0^s f(r) dr \quad \text{and} \quad \mathcal{J}_f(s_1,s_2) \circeq \int_{s_1}^{s_2} f(r) dr\ .
    \end{equation*}
    It turns out that the first order of the maximum and the log-number of high points are controlled by the {\it concavification} of $\mathcal{J}_{\sigma^2}(\cdot)$.
    Let $\hat{\mathcal{J}}_{\sigma^2}$ be the function whose graph is the concave hull of $\mathcal{J}_{\sigma^2}$.
    Its graph is an increasing and concave polygonal line, see Figure \ref{fig:concave} for an example.
    There exists a unique non-increasing left-continuous step function $s \mapsto \bar{\sigma}(s)$ such that
    \begin{equation*}
        \hat{\mathcal{J}}_{\sigma^2}(s) = \mathcal{J}_{\bar{\sigma}^2}(s) = \int_0^s \bar{\sigma}^2(r) dr \  \text{ for all $s\in (0,1]$\hspace{0.3mm}. }
    \end{equation*}
    The points on $[0,1]$ where $\bar{\sigma}$ jumps will be denoted by
    \begin{equation}
        \label{eqn: lambda jumps}
        0 \circeq \lambda^0 < \lambda^1 < ... < \lambda^m \circeq 1,
    \end{equation}
    where $m \leq M$. To be consistent with previous notations, we set $\bar{\sigma}_l \circeq \bar{\sigma}(\lambda^l)$.
    \begin{figure}[ht]
        \centering
        \includegraphics[scale=0.7]{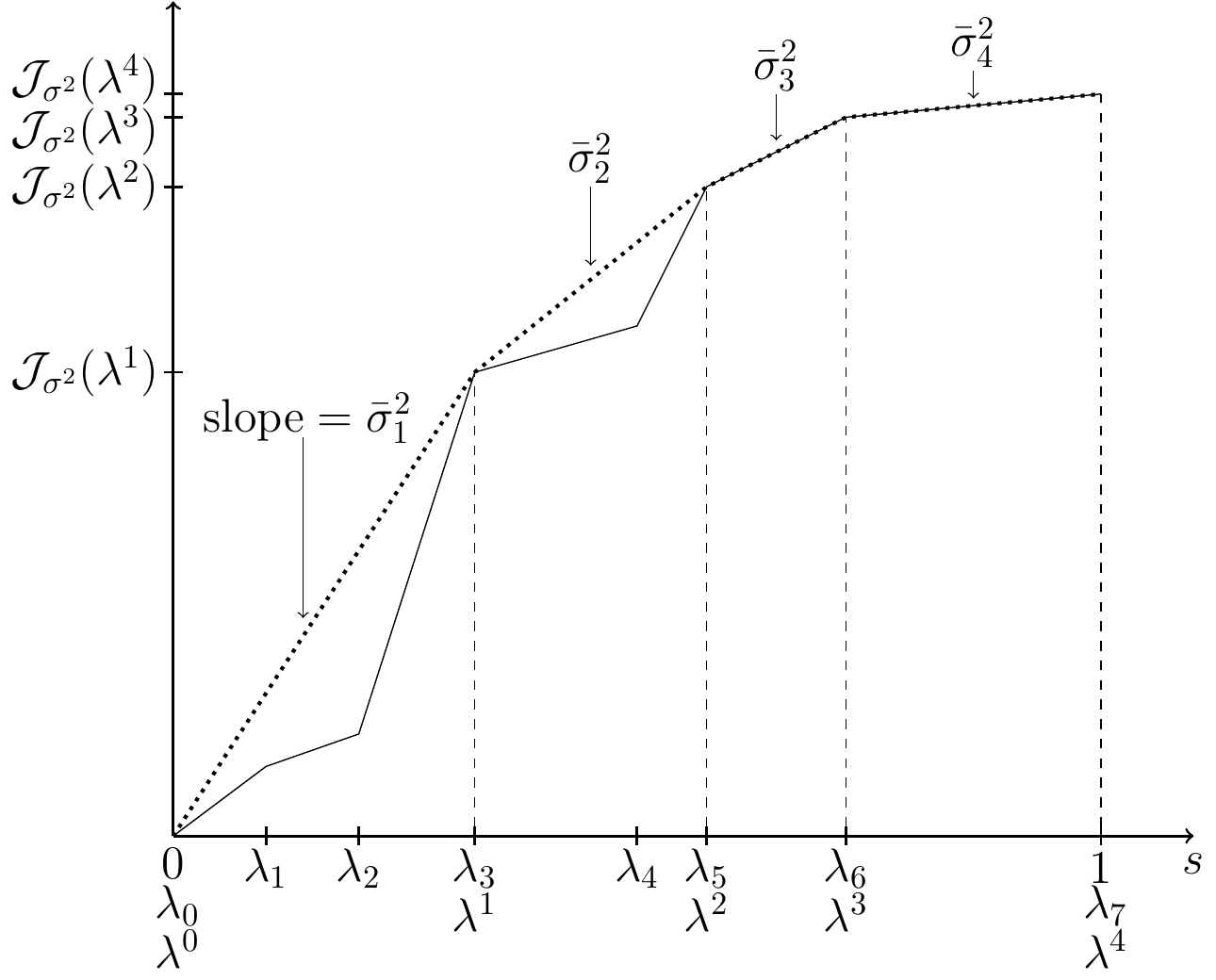}
        \captionsetup{width=0.8\textwidth}
        \caption{Example of $\mathcal{J}_{\sigma^2}$ (closed line) and $\hat{\mathcal{J}}_{\sigma^2}$ (dotted line) with $7$ values for $\sigma^2$.\vspace{-2mm}}
        \label{fig:concave}
    \end{figure}

    \begin{theorem}[First order of the maximum]
        \label{thm:IGFF.order1}
        Let $\{\psi_v\}_{v\in V_N}$ be the $(\boldsymbol{\sigma},\boldsymbol{\lambda})$-GFF on $V_N$ of Definition \ref{def:IGFF}, then
        \begin{equation*}
            \lim_{N\rightarrow\infty} \frac{\max_{v\in V_N} \psi_v}{\log N^2} = \mathcal{J}_{\sigma^2/\bar{\sigma}}(1)\circeq \gamma^\star\quad \text{in probability}\hspace{0.3mm}.
        \end{equation*}
    \end{theorem}
    In the homogeneous case where $M = 1$ and $\sigma_1 = 1$, the result reduces to $\gamma^{\star} = 1$, as proved in \citet{MR1880237}, which corresponds to the first order of the maximum of $N^2$ i.i.d.~Gaussian variables of mean $0$ and variance $\log N$.
    Note that the result of Theorem \ref{thm:IGFF.order1} can be written as follows :
    \vspace{1mm}
    \begin{equation}\label{eq:order1.maximum.continuous.GFF}
        \gamma^\star = \mathcal{J}_{\sigma^2/\bar{\sigma}}(1) = \sum_{l=1}^m \int_{\lambda^{l-1}}^{\lambda^l}\frac{\sigma^2(s)}{\bar{\sigma}(s)}ds = \int_0^1 \bar{\sigma}(s)ds\ .
    \end{equation}
    This is simply a weighted average of homogeneous cases on the intervals $[\lambda^{l-1},\lambda^l]$ with variance parameter $\bar{\sigma}_l$.
    We say that $s \mapsto \bar{\sigma}^2(s)$ act as the \textit{effective variance} of the field.
    We stress that $\gamma^\star$ is strictly smaller than $\bar{\sigma}_1$ in cases where the concave hull is not a straight line.
    In particular, the upper bound on the level of the maximum cannot be proved by a simple union bound as in the homogeneous case.

    The set of $\gamma$-high points of the field $\psi$ is defined as
    \begin{equation*}
        \mathcal{H}_N^{\gamma} \circeq \{v\in V_N \nvert \psi_v \geq \gamma \log N^2\},\ \text{ for all $0 \leq \gamma < \gamma^{\star}$.}
    \end{equation*}
    The number of high points will depend on {\it critical levels} defined by
    \begin{equation}\label{eq:critic.levels}
        \gamma^{l}
        \circeq \int_0^1 \frac{\sigma^2(s)}{\bar{\sigma}(s \wedge \lambda^l)} ds
        = \mathcal{J}_{\sigma^2 / \bar{\sigma}}(\lambda^l) + \frac{\mathcal{J}_{\sigma^2}(\lambda^l,1)}{\bar{\sigma}_l},  \ \ 1\leq l\leq m,
        \quad \gamma^0\circeq 0 \hspace{0.3mm}.
    \end{equation}

    \begin{theorem}[Log-number of high points or Entropy]\label{thm:IGFF.high.points}
        Let $\{\psi_v\}_{v\in V_N}$ be the $(\boldsymbol{\sigma},\boldsymbol{\lambda})$-GFF on $V_N$ of Definition \ref{def:IGFF} and let $\gamma^{l-1} \leq \gamma < \gamma^{l}$ for some $l\in \{1,...,m\}$, then
        \begin{equation*}
            \lim_{N\rightarrow\infty} \frac{\log |\mathcal{H}_N^{\gamma}|}{\log N^2} =
            (1 - \lambda^{l-1}) - \frac{(\gamma - \mathcal{J}_{\sigma^2 / \bar{\sigma}}(\lambda^{l-1}))^2}{\mathcal{J}_{\sigma^2}(\lambda^{l-1},1)}
            \circeq \mathcal{E}_{\gamma}
            \quad
            \text{in probability\hspace{0.3mm}.}
        \end{equation*}
    \end{theorem}
    The homogeneous case where $M=1$ and $\sigma_1 = 1$ was proved in \citet{MR2243875}.
    In that case, we have $\mathcal{E}_{\gamma}=1-\gamma^2$ as for $N^2$ i.i.d.~Gaussian variables of mean \hspace{-0.2mm}$0$ and \hspace{-0.2mm}variance \hspace{-0.2mm}$\log N$.
    \hspace{-0.4mm}The proofs of Theorems \hspace{-0.2mm}\ref{thm:IGFF.order1} \hspace{-0.2mm}and \hspace{-0.2mm}\ref{thm:IGFF.high.points} \hspace{-0.2mm}are deferred to Section \ref{sec:main.results.proof}.
    The method of proof is explained in Section \ref{sec:outline.proofs}.
    It is a refinement of the second moment method based on the control of the increments of high points at every scale.
    The method was used in \citet{MR3380419} to obtain a new proof of the first order of the maximum in the homogeneous case.
    Here we extend this method to the log-number of high points in all settings and to the first order of the maximum in the inhomogeneous setting.
    In the scale-dependent case, as opposed to the homogeneous case, it is necessary to truncate the first moment using the information at every scale $\lambda^l$ to get the correct upper bound.

    \subsection{Related works and conjectures}\label{sec:related.work.conjectures}

    The scale-inhomogeneous GFF is the equivalent of the time-inhomogeneous branching random walk (IBRW) where the variance of the random walk is a function of time.
    In particular, Theorems \ref{thm:IGFF.order1} and \ref{thm:IGFF.high.points} can be proved for branching random walks using the same technique, see Section $2$ of \citet{Ouimet2014master}.
    In fact, much more precise information is known about the maxima of these models.
    In \citet{MR2070335}, the authors introduce a continuous version of Derrida's {\it Generalized Random Energy Model} (GREM) \citet{Derrida85}, which is akin to a time-inhomogeneous branching random walk, for which they obtain the first order of the maximum and the free energy.
    In particular, they noticed the concavification phenomenon for the first order.
    This observation also appears in \citet{MR883541} for the GREM.
    A model interpolating between the GREM and the branching random walk was introduced in \citet{MR3358969} where Poisson statistics of the extremes are proved.
    For Gaussian IBRWs with two values of the variance ($M=2$), the lower order corrections for the maximum and tightness of the law were proved in \citet{MR2968674}.
    In this case, convergence of the extremal processes and of the law of the recentered maximum have been shown in \citet{MR3164771}.
    This is also proved in the case where the integral of the variance remains strictly below its concave hull (for example, in the case of increasing variances), see \citet{MR3351476}.
    For strictly decreasing variances, the lower order corrections for IBBMs exhibit a slowdown of the order $t^{1/3}$ as proved in \citet{MR2981635,MR3531703}.
    Similar results for non-Gaussian IBRWs and more general variances are proved in \citet{MR3373310}, though not at the level of convergence of the law.
    In \citet{arXiv:1509.08172}, the second order of the maximum for the Gaussian IBRW with a finite number of variances is shown by generalizing the approach of \citet{MR2968674} and the tightness follows from \citet{MR3012090}.

    In general, we expect that the scale-inhomogeneous GFF with a finite number of variances behave as the time-inhomogeneous branching random walk with the same parameters for the lower order correction term of the maximum and for its law.
    For the homogeneous GFF, the convergence of the law of the recentered maximum was proved in \citet{MR3433630}.
    \hspace{-0.3mm}In \hspace{-0.3mm}\citet{MR3354619}, \hspace{-0.3mm}the scale-inhomogeneous GFF with two values of the variance was introduced to prove Poisson-Dirichlet statistics for the extremes of the homogeneous GFF.
    Actual Poisson statistics for local extremes was proved later in \citet{MR3509015}.

    One interest of Definition \ref{def:IGFF} for the scale-inhomogeneous GFF is that it can be extended to a piecewise smooth variance function $\sigma : [0,1] \to [a,b]$ where $a > 0$.
    Consider the two-dimensional continuous Gaussian free field $\phi = \{\phi_v\}_{v\in[0,1]^2}$ on the unit square $[0,1]^2$, see e.g. \citet{MR2322706} for a definition.
    The field $\phi$ cannot be defined as a random function. However, averages over sets make sense as random variables.
    In particular, for every $v\in[0,1]^2$ and $\lambda\in [0,1]$, one can define $\phi_v^r(\lambda)$ as the average of the field over a circle of radius $r^{\lambda}$ :
    \begin{equation}\label{eq:continuous.GFF.conditional.expectation}
        \phi_v^r(\lambda) \circeq \frac{1}{2\pi r^{\lambda}}\int_0^{2\pi} \phi_{v+ r^{\lambda} e^{i\theta}}\ d\theta\ .
    \end{equation}
    The parameter $r$ plays the role of $N^{-1}$ in the discrete setting. The continuous scale-inhomogeneous GFF for the variance function $\lambda\mapsto \sigma(\lambda)$ can then be defined in terms of these averages :
    \vspace{-1.5mm}
    \begin{equation*}
        \psi_v^r(1)\circeq \int_0^1 \sigma(\lambda) ~d\phi_v^r(\lambda), \ \ \ v\in [0,1]^2.
    \end{equation*}
    The stochastic integral makes sense because $(\phi_v^r(\lambda))_{\lambda\in [0,1]}$ is a Gaussian martingale.
    Following the definition in \citet{MR2819163} (up to a factor $2$), a point $v\in [0,1]^2$ is called {\it $\gamma$-thick} if
    \vspace{-1mm}
    \begin{equation*}
        \lim_{r\to 0} \frac{\psi_v^r(1)}{\log (r^{-2})}\geq \gamma
    \end{equation*}
    where it is assumed that the continuous Green function on \hspace{-0.15mm}$[0,1]^2$ \hspace{-0.2mm}associated to \hspace{-0.1mm}$\phi$ has been normalized as in \eqref{eqn: G}.
    This is the notion analogous to $\gamma$-high points.
    It was shown in \citet{MR2642894} that the Hausdorff dimension of the set of $\gamma$-thick points is $2(1-\gamma^2)$ when $\sigma \equiv 1$.
    In view of Theorem \ref{thm:IGFF.high.points}, it is reasonable to conjecture that the Hausdorff dimension of the set of $\gamma$-thick points of $\psi$ is
    \begin{equation}\label{eq:entropy.continuous.GFF}
        2\left((1 - \lambda_{\star}) - \frac{(\gamma - \mathcal{J}_{\sigma^2 / \bar{\sigma}}(\lambda_{\star}))^2}{\mathcal{J}_{\sigma^2}(\lambda_{\star},1)}\right),
    \end{equation}
    where $\lambda_{\star} \circeq \inf\{\lambda\in [0,1] : \gamma \leq \int_0^1 \frac{\sigma^2(s)}{\bar{\sigma}(s \wedge \lambda)} ds\}$.

    \section{Outline of Proof}\label{sec:outline.proofs}

    As stated before, the results of this paper are applicable to time-inhomogeneous branching random walks and, more generally, to any scale-dependent log-correlated Gaussian field.
    The proof relies on two main ingredients: an underlying approximate tree structure present in log-correlated models and an adaptation of the multiscale refinement of the second moment method introduced in \citet{MR3380419}. In particular, the method requires understanding the increments of high points along every scale to prove tight upper and lower bounds. In \citet{MR3380419}, this method was used to streamline the proof of \citet{MR1880237} for the first order of the maximum of the homogeneous GFF.
    Here, we adapt the method to deal with scale-inhomogeneous fields and log-number of high points.

    To see the tree structure, define the {\it branching scale between $v$ and $v'$} in $V_N$ :
    \begin{equation}\label{eq:rho}
        \rho(v,v')\circeq \max\{\lambda\in [0,1] : [v]_\lambda \cap [v']_\lambda \neq \emptyset\}\ .
    \end{equation}
    This is the largest $\lambda$ for which the two neighborhoods $[v]_\lambda$ and $[v']_\lambda$ intersect.
    We always have by definition that $\|v-v'\|_2$ is of order $N^{1-\rho(v,v')}$.
    The branching scale plays the same role as the branching time (normalized to lie in $[0,1]$) in branching random walk.
    More precisely, let $\{\phi_v\}_{v\in V_N}$ be a homogeneous GFF and consider the increments $\phi_v(\lambda')- \phi_v(\lambda)$ and $\phi_{v'}(\mu')- \phi_{v'}(\mu)$ for some choice of $\lambda < \lambda'$ and $\mu < \mu'$.
    The Markov property of the Gaussian free field (see Section \ref{sec:tech.lemmas}) implies that for $\lambda,\mu > \rho(v,v')$,
    \begin{equation*}
        \phi_v(\lambda')- \phi_v(\lambda) \quad \text{is independent of} \quad \phi_{v'}(\mu')- \phi_{v'}(\mu),
    \end{equation*}
    because the neighborhoods $[v]_{\lambda}$ and $[v']_{\mu}$ are disjoint, see Figure \ref{fig:branching}.
    This means that the increments after the branching scale are independent.

    On the other hand, if $\lambda < \rho$, it can be shown using Green function estimates (see e.g. Lemma 12  in \citet{MR1880237}) that
    \begin{equation*}
        \var{}{\phi_v(\lambda)- \phi_{v'}(\lambda)} = O(1)\ .
    \end{equation*}
    In other words, the values of $\phi_v(\lambda)$ and $\phi_{v'}(\lambda)$ must be close.
    This suggests that the increments before the branching scale are almost identical.
    In particular, without losing much information, we can restrict the field $\{\phi_v(\lambda)\}_{v\in V_N}$ to a set $R_\lambda \subseteq V_N$ containing $\lfloor N^{\lambda} \rfloor^2$ $v$'s with neighborhoods $[v]_{\lambda}$ that can only touch at their boundary and are not cut off by $\partial V_N$.
    To remove any ambiguity, define $R_{\lambda}$ in such a way that $\max_{v\in V_N} \min_{z\in R_{\lambda}} \|v - z\|_2$ is minimum.
    We call $R_{\lambda}$ the {\it set of representatives at scale $\lambda$} and define $R_1 \circeq V_N$.
    For instance, if $N = 2^n$, $\lambda\in [0,1)$ and $\lambda n\in \N$, then divide $V_N$ into a grid with $N^{2\lambda}$ squares of side length $N^{1-\lambda}$, the center point of each square is a representative at scale $\lambda$.

    \begin{figure}[ht]
        \centering
        \includegraphics[scale=0.65]{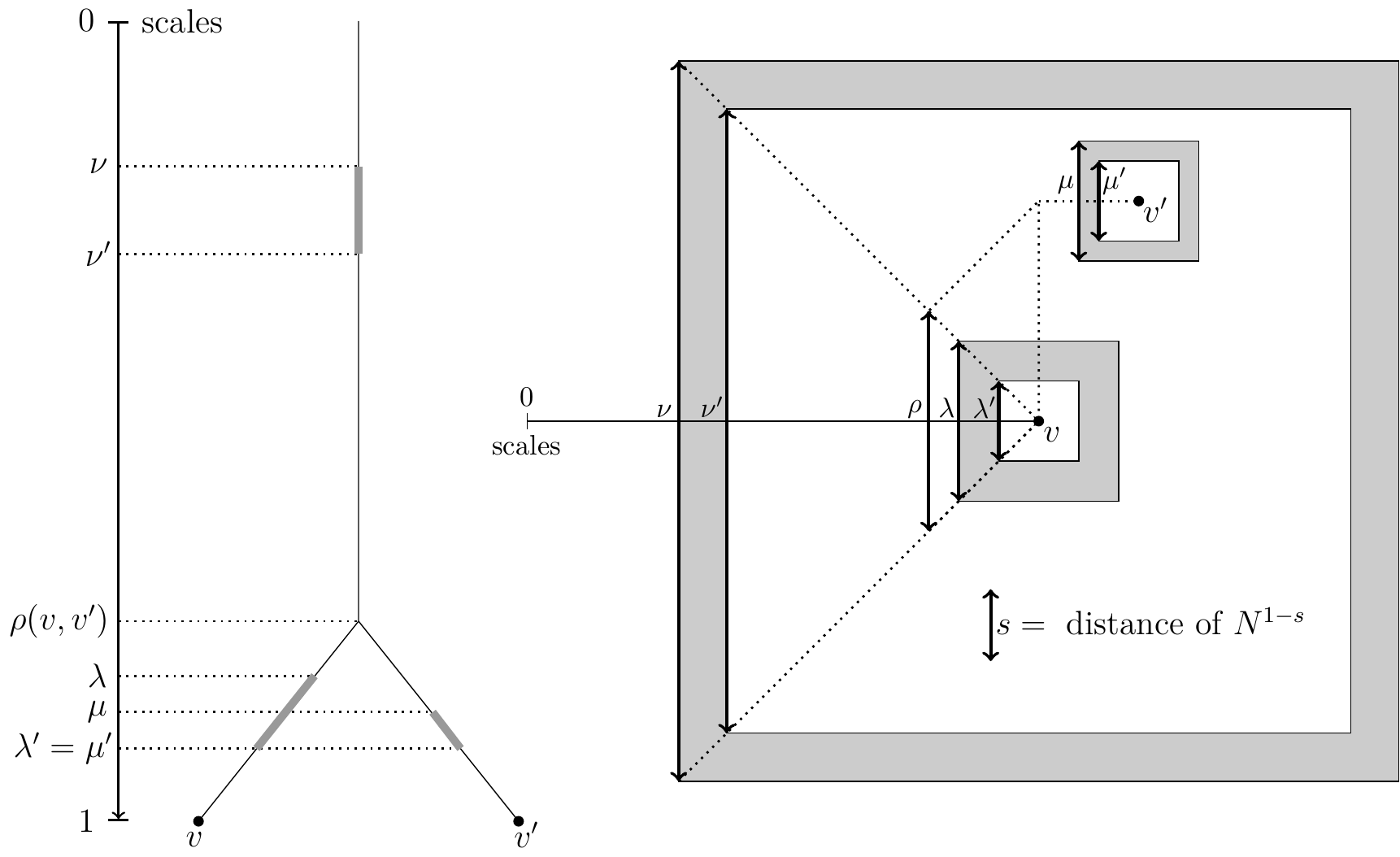}
        \captionsetup{width=0.85\textwidth}
        \caption{The branching structure of the GFF.}
        \label{fig:branching}
    \end{figure}

    Of course, the branching structure here is not exact as in branching random walk.
    In particular, nothing precise can be said on the increments $\phi_v(\lambda') - \phi_v(\lambda)$ and $\phi_{v'}(\lambda') - \phi_{v'}(\lambda)$ in the case where $\lambda < \rho < \lambda'$.
    However, the contribution of such increments can be made negligible by considering a large number of increments, as we shall do.
    This branching structure holds also for the $(\boldsymbol{\sigma},\boldsymbol{\lambda})$-GFF, since it is defined in terms of the increments of $\phi$, see \eqref{eqn: psi v} and Lemma \ref{lem:psi.Markov}.

    For $0 < \gamma < \gamma^{\star}$, the $\gamma$-high points are such that $\psi_v \geq \gamma \log N^2$.
    It is reasonable to expect that for these points, there exists a unique optimal path $\lambda \mapsto L_N^{\gamma}(\lambda)$ such that $\psi_v(\lambda) \geq L_N^{\gamma}(\lambda)$ at each scale $\lambda$.
    We write $L_N^\star$ for the corresponding optimal path in the case of the maximum level $\gamma^\star$.
    It is the information on these paths  along the scales that is crucial for the method to yield tight upper and lower bounds.
    We explain heuristically how to determine these optimal paths using first moments.

    Consider the set of $v$'s for which the increments of the field $\psi$ reach level $\nabla \gamma_i$ between each scale $\lambda_i$ :
    \begin{equation*}
        \Lambda_{N,M} \circeq \{v\in V_N \nvert \nabla \psi_v(\lambda_i) \geq \nabla \gamma_i \log N^2 ~\text{for all } i\in\{1,2,\dots, M\}\}\ ,
    \end{equation*}
    where $\gamma_0 \circeq 0$. By construction, $|\Lambda_{N,M}|$ is a lower bound on the number of points in $V_N$ reaching a height of $\gamma_M \log N^2$.
    We also consider the corresponding quantity at intermediate scales $\lambda_k < \lambda_M$.
    In this case, because of correlations, we can restrict ourselves to representatives at scale $\lambda_k$ :
    \begin{equation*}
        \Lambda_{N,k} \circeq \{v\in R_{\lambda_k} \nvert \nabla \psi_v(\lambda_i) \geq \nabla \gamma_i \log N^2 ~\text{for all } i\in\{1,2,\dots, k\}\}\ .
    \end{equation*}\vspace{0.5mm}

    There are $O(N^{2\lambda_k})$ representatives at scale $\lambda_k$ and the variance of the increments is $\mathbb{V}(\nabla \psi_v(\lambda_i)) \hspace{-0.5mm}= \sigma_i^2 \nabla \lambda_i \log N + O(1)$ if we ignore the boundary effect.
    Therefore, using the independence between the increments and standard Gaussian estimates (see Lemma \ref{lem:gaussian.estimates}, it will be used repeatedly) :
    \vspace{-2mm}
    \begin{equation*}
        \esp{}{|\Lambda_{N,k}|} \asymp N^{2 \lambda_k} \prod_{i=1}^k \prob{}{\nabla \psi_v(\lambda_i) \geq 2 \nabla \gamma_i \log N} \asymp \frac{N^{2 \lambda_k} N^{-2 \sum_{i=1}^k \frac{(\nabla \gamma_i)^2}{\sigma_i^2 \nabla \lambda_i}}}{(\log N)^{k/2}}\ ,
    \end{equation*}
    where $\asymp$ means that the ratio of the two sides lies in a compact interval bounded away from $0$, for $N$ large enough.
    In other words,
    \begin{equation*}
        \lim_{N\rightarrow\infty} \frac{\log(\esp{}{|\Lambda_{N,k}|})}{\log N^2} = \sum_{i=1}^k \left(\nabla \lambda_i - \frac{(\nabla \gamma_i)^2}{\sigma_i^2 \nabla \lambda_i}\right).
    \end{equation*}
    Since there should be representatives at each scale $\lambda_k$ that ultimately yield a high value at scale $\lambda_M$,
    it is intuitive that the level of the maximum can be found by maximizing
    \begin{equation*}
        \gamma_M = \sum_{i=1}^M \nabla \gamma_i \ \ \text{under the constraints} \ \ \sum_{i=1}^k \left(\nabla \lambda_i - \frac{(\nabla \gamma_i)^2}{\sigma_i^2 \nabla \lambda_i}\right) \geq 0\ , \ \ \ 1 \leq k \leq M\hspace{0.3mm}.
    \end{equation*}
    This optimization problem can be solved using the Karush-Kuhn-Tucker theorem (see Lemma \ref{lem:optimization.1}).
    We write $(\gamma_1^{\star},\gamma_2^{\star},...,\gamma_M^{\star})$ for the unique solution.
    We will make extensive use of the polygonal line $L_N^{\star}(\cdot)$ linking the points $(0,0)$, $(\lambda_1,\gamma_1^{\star} \log N^2)$, $(\lambda_2,\gamma_2^{\star} \log N^2)$, \ldots, $(\lambda_M,\gamma_M^{\star} \log N^2)$ to prove Theorem \ref{thm:IGFF.order1} and \ref{thm:IGFF.high.points} :
    \begin{equation}\label{eqn: LN star}
        L_N^{\star}(s)\circeq \int_0^s \frac{\sigma^2(r)}{\bar{\sigma}(r)} dr \log N^2 = \mathcal{J}_{\sigma^2 / \bar{\sigma}}(s) \log N^2, \quad s\in [0,1]\hspace{0.3mm}.
    \end{equation}
    This is the \textit{optimal path for the maximum}. Figure \ref{fig:optimal} shows an example of such a path.
    In particular, it is important to note that the optimal path coincides with its concave hull at each scale $\lambda^l$, namely
    \begin{equation}
        \label{eqn: LN points}
        L_N^{\star}(\lambda^l) = \hat{L}_N^{\star}(\lambda^l) = \hat{\mathcal{J}}_{\sigma^2 / \bar{\sigma}}(\lambda^l) \log N^2 = \mathcal{J}_{\bar{\sigma}}(\lambda^l) \log N^2, \quad 1 \leq l \leq m\ .
    \end{equation}

    \begin{figure}[ht]
        \centering
        \includegraphics[scale=0.55]{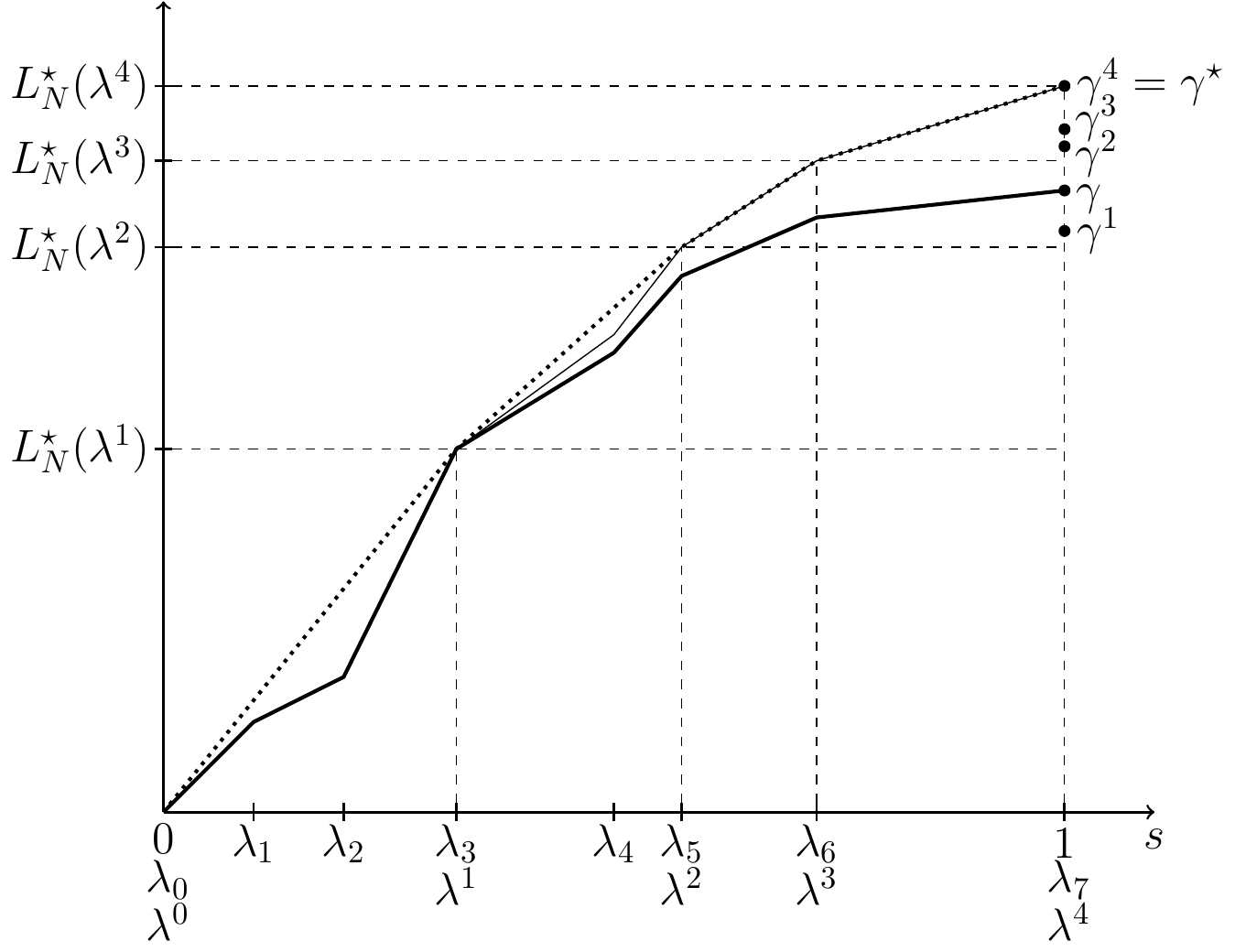}
        \captionsetup{width=0.85\textwidth}
        \caption{Example of $L_N^{\gamma}$ (bold line), $L_N^{\star}$ (thin line) and its concavified version $\hat{L}_N^{\star}$ (dotted line), with $7$ values for $\sigma^2$ and $\gamma^1 < \gamma < \gamma^2$.}
        \label{fig:optimal}
    \end{figure}
    The same heuristic can be used to determine the optimal path $L_N^{\gamma}(\cdot)$ for $\gamma$-high points, $0 < \gamma < \gamma^{\star}$.
    Setting now $\gamma_M=\gamma$, we get
    \begin{equation}\label{eq:energy.high.points}
        \lim_{N\rightarrow\infty} \frac{\log(\esp{}{|\Lambda_{N,M}|})}{\log N^2} = \sum_{i=1}^{M-1} \left(\nabla \lambda_i - \frac{(\nabla \gamma_i)^2}{\sigma_i^2 \nabla \lambda_i}\right) + \left(\nabla \lambda_M - \frac{(\gamma - \gamma_{M-1})^2}{\sigma_M^2 \nabla \lambda_M}\right).
    \end{equation}
    A \hspace{-0.1mm}lower \hspace{-0.1mm}bound \hspace{-0.1mm}for \hspace{-0.1mm}the \hspace{-0.1mm}log-number \hspace{-0.1mm}of \hspace{-0.1mm}$\gamma$-high \hspace{-0.1mm}points \hspace{-0.1mm}can \hspace{-0.1mm}be \hspace{-0.1mm}found \hspace{-0.1mm}by \hspace{-0.1mm}maximizing \hspace{-0.1mm}\eqref{eq:energy.high.points} with respect to $\gamma_1, \gamma_2, \ldots,\gamma_{M-1}$ and under the constraints
    \begin{equation}\label{eq:constraints.optimization.high.points}
        \sum_{i=1}^k \left(\nabla \lambda_i - \frac{(\nabla \gamma_i)^2}{\sigma_i^2 \nabla \lambda_i}\right) \geq 0, \ \ \ 1 \leq k \leq M-1\ .
    \end{equation}
    The unique solution to this problem is found in Lemma \ref{lem:optimization.2} using again the Karush-Kuhn-Tucker theorem.
    The form of the path will always depend on the critical levels defined in \eqref{eq:critic.levels}. Whenever $\gamma^{l-1} \leq \gamma < \gamma^l$, the \textit{optimal path for $\gamma$-high points} is :
    \begin{equation}\label{eqn: gamma optimal}
        \hspace{2mm}
        L_N^{\gamma}(s) \circeq
        \left\{\hspace{-1.5mm}
        \begin{array}{ll}
            \mathcal{J}_{\sigma^2 / \bar{\sigma}}(s) \log N^2, & \hspace{-1.5mm}0 \leq s \leq \lambda^{l-1} \\
            \left(\mathcal{J}_{\sigma^2 / \bar{\sigma}}(\lambda^{l-1}) + \frac{\mathcal{J}_{\sigma^2}(\lambda^{l-1},s)}{\mathcal{J}_{\sigma^2}(\lambda^{l-1},1)} (\gamma - \mathcal{J}_{\sigma^2 / \bar{\sigma}}(\lambda^{l-1}))\right) \log N^2, & \hspace{-1.5mm}\lambda^{l-1} \leq s \leq 1.
        \end{array}
        \right.
    \end{equation}
    The path coincide on $[0,\lambda^{l-1}]$ with the optimal path for the maximum.
    Also, note that $L_N^{\gamma}$ is continuous and converges uniformly to $L_N^{\star}$ as $\gamma \rightarrow \gamma^{\star}$ (which yields that $L_N^{\star}$ is continuous as well).

\section{Proofs of the main results}\label{sec:main.results.proof}
    \subsection{\hspace{-0.5mm}Preliminaries}

        For all $\lambda\in [0,1]$, recall that $\psi_v(\lambda) \circeq \mathbb{E}[\psi_v \nvert \F_{\partial [v]_{\lambda} \cup [v]_{\lambda}^c}]$.
        By the Markov property of the GFF (see Lemma \ref{lem:psi.Markov}), it is not hard to show that for any partition $0 \circeq s_0 < s_1 < ... < s_K \circeq 1$ of $[0,1]$ such that $\{\lambda_i\}_{i=0}^M \subseteq \{s_j\}_{j=0}^K$, we have for all $1 \leq k \leq l \leq K$ :
        \vspace{-2mm}
        \begin{equation*}
            \psi_v(s_l) - \psi_v(s_{k-1}) = \sum_{j=k}^l \sigma(s_j) \nabla \phi_v(s_j) \ .
        \end{equation*}
        In particular, the independence of the increments of $\psi$ follows directly from the one for $\phi$.
        Moreover, using standard estimates on Green functions, Lemma \ref{lem:tech.lemma.1} shows that
        \begin{equation}\label{eq:borne.variance.delta}
            -C_1(\delta) \leq \var{}{\psi_v(s_l) - \psi_v(s_{k-1})} - \mathcal{J}_{\sigma^2}(s_{k-1},s_l) \log N \leq C_2
        \end{equation}
        for all $v\in V_N^{\delta}$ and $N$ large enough (depending on $\delta$), where
        \begin{equation*}
            V_N^\delta \circeq \{v\in V_N ~| \min_{z\in \partial V_N} \|v - z\|_2 \geq \delta N\}, \ \ \ \delta \in (0,1/2]\ .
        \end{equation*}
        The set $V_N^{\delta}$ contains the points that are at a distance at least $\delta N$ from the boundary of $V_N$. Lemma \ref{lem:tech.lemma.2} proves that the upper bound in \eqref{eq:borne.variance.delta} holds on $V_N$, that is
        \begin{equation}\label{eq:borne.variance.upper}
            \max_{v\in V_N} \var{}{\psi_v(s_l) - \psi_v(s_{k-1})} \leq \mathcal{J}_{\sigma^2}(s_{k-1},s_l) \log N ~+ ~C
        \end{equation}
        for $N$ large enough.

        Throughout the proofs, $c$ and $C$ will denote positive constants whose value can change at different occurrences and might depend on the parameters $(\boldsymbol{\sigma},\boldsymbol{\lambda})$. For simplicity, equations in the proofs are implicitly stated to hold for $N$ large enough where it is needed.

    \subsection{\hspace{-0.5mm}First order of the maximum}

        \hspace{-0.5mm}Theorem \ref{thm:IGFF.order1} is a direct consequence of Lemma \ref{lem:IGFF.order1.upper.bound}, which proves that $\gamma^{\star}\log N^2$ is an upper bound on the first order of the maximum, and Lemma  \ref{lem:IGFF.order1.lower.bound} which shows the corresponding lower bound.

        \begin{lemma}[Upper bound on the first order of the maximum]\label{lem:IGFF.order1.upper.bound}
            Let $\{\psi_v\}_{v\in V_N}$ be the $(\boldsymbol{\sigma},\boldsymbol{\lambda})$-GFF on $V_N$ of Definition \ref{def:IGFF} and $\gamma^{\star}$ as in Theorem \ref{thm:IGFF.order1}. For all $\varepsilon > 0$, there exists a constant $c = c(\varepsilon,\boldsymbol{\sigma},\boldsymbol{\lambda}) > 0$ such that
            \begin{equation}\label{eq:IGFF.order1.upper.bound}
                \prob{}{\max_{v\in V_N} \psi_v \geq (\gamma^{\star} + m \varepsilon) \log N^2} \leq N^{-c}
            \end{equation}
            for $N$ large enough.
        \end{lemma}

        \begin{proof}
            Recall the definition of the optimal path $L_N^\star$ from \eqref{eqn: LN star} and define
            \begin{equation*}
                L_N^{\star,z}(s) \circeq L_N^{\star}(s) + z \log N^2, \ \ \ s\in [0,1]\ .
            \end{equation*}
            Recall the definition of $\lambda^j$ in \eqref{eqn: lambda jumps} and the notation $R_{\lambda^j}$ for the set of representatives at scale $\lambda^j$.
            Consider the set of representatives whose value reached just over the optimal level at $\lambda^j$ :
            \begin{equation*}
                \mathcal{H}_{N,j}^{\star, \varepsilon} \circeq \left\{v\in R_{\lambda^j} \nvert \psi_v(\lambda^j) \geq L_N^{\star,j\varepsilon}(\lambda^j)\right\}, \ \ \ 1 \leq j \leq m\ .
            \end{equation*}
            The idea of the proof is to split the probability that at least one point in $V_N$ reaches just over the optimal height by looking at the first scale $\lambda^j, ~1 \leq j \leq m$, where the set $\mathcal{H}_{N,j}^{\star,\varepsilon}$ is not empty. This provides the appropriate constraints along the scales to get the correct upper bound.
            For $0 < \eta_{\varepsilon} < \varepsilon/m$, define
            \begin{equation*}
                A_{\varepsilon} \circeq \left\{|\psi_v(\lambda^j) - \psi_{v_{\lambda^j}}(\lambda^j)| \leq \eta_{\varepsilon} \log N^2 ~~\text{for all } j\in \{1,...,m\} ~\text{and all } v\in V_N\right\}
            \end{equation*}
            where $v_{\lambda}$ denotes any representative in $R_{\lambda}$ that is closest to $v$.
            Here we introduced the event $A_{\varepsilon}$ to approximate the branching structure of the field $\psi$.
            Since $R_{\lambda^m} = V_N$ by definition and $L_N^{\star,m\varepsilon}(\lambda^m) = (\gamma^{\star} + m\varepsilon) \log N^2$, a union bound gives the following upper bound on the probability in \eqref{eq:IGFF.order1.upper.bound} :
            \begin{align}\label{eqn: H split.1}
                    \mathbb{P}\hspace{-0.5mm}\left(|\mathcal{H}_{N,m}^{\star,\varepsilon}| \geq 1\right)
                    &\leq \mathbb{P}\hspace{-0.5mm}\left(A_{\varepsilon}^c\right) + \sum_{l=1}^m \mathbb{P}\hspace{-0.5mm}\left(\hspace{-0.5mm}\left\{|\mathcal{H}_{N,1}^{\star,\varepsilon}| = ... = |\mathcal{H}_{N,l-1}^{\star,\varepsilon}| = 0, |\mathcal{H}_{N,l}^{\star,\varepsilon}| \geq 1\right\} \hspace{-0.3mm}\cap \hspace{-0.3mm}A_{\varepsilon}\hspace{-0.3mm}\right) \notag \\
                    &\leq \mathbb{P}\hspace{-0.5mm}\left(A_{\varepsilon}^c\right) + \sum_{l=1}^m \mathbb{P}\hspace{-0.5mm}\left(\hspace{-1mm}
                    \left\{\hspace{-1mm}
                    \begin{array}{l}
                        \exists v\in R_{\lambda^l} \cap V_N^o ~~\text{s.t.} \\
                        \psi_v(\lambda^l) \geq L_N^{\star,l\varepsilon}(\lambda^l) ~\text{and} \\
                        \psi_{v_{\lambda^j}}(\lambda^j) < L_N^{\star,j\varepsilon}(\lambda^j) \\
                        \text{for all } 1 \leq j \leq l-1
                    \end{array}\hspace{-1mm}
                    \right\}
                    \cap A_{\varepsilon}
                    \right) \notag \\
                    &\leq C e^{-c(\eta_{\varepsilon}) (\log N)^2} \notag \\
                    &\quad+ \sum_{l=1}^m N^{2 \lambda^l} \hspace{-2mm}\max_{v\in R_{\lambda^l} \cap V_N^o} \hspace{-1mm}\mathbb{P}\hspace{-0.5mm}\left(\hspace{-1mm}
                    \left\{\hspace{-1mm}
                    \begin{array}{l}
                        \psi_v(\lambda^l) \geq L_N^{\star,l\varepsilon - \eta_{\varepsilon}}(\lambda^l) ~\text{and} \\
                        \psi_v(\lambda^j) < L_N^{\star,j(\varepsilon + \eta_{\varepsilon})}(\lambda^j) \\
                        \text{for all } 1 \leq j \leq l-1
                    \end{array}\hspace{-1mm}\right\}
                    \hspace{-1mm}\right)
            \end{align}
            The bound on $\prob{}{A_{\varepsilon}^c}$ follows easily from a union bound (with $m \cdot (N+1)^2$ terms), Gaussian estimates (Lemma \ref{lem:gaussian.estimates}) and the variance estimates of Lemma \ref{lem:tech.lemma.3}.

            It remains to consider the terms in the sum in \eqref{eqn: H split.1}. We look at the case $l=1$. Since $\max_{v\in V_N} \var{}{\psi_v(\lambda^1)} \leq \lambda^1 \bar\sigma_1^2 \log N + C$ from \eqref{eq:borne.variance.upper} and $L_N^{\star}(\lambda^1) = \lambda^1 \bar\sigma_1 \log N^2$, a Gaussian estimate shows that
            \begin{align*}
                \prob{}{\psi_v(\lambda^1) \geq L_N^{\star,\varepsilon - \eta_{\varepsilon}}(\lambda^1)}
                &\leq \frac{\sqrt{\var{}{\psi_v(\lambda^1)}}}{L_N^{\star,\varepsilon - \eta_{\varepsilon}}(\lambda^1)} \exp\left(-\frac{(L_N^{\star,\varepsilon - \eta_{\varepsilon}}(\lambda^1))^2}{2\var{}{\psi_v(\lambda^1)}}\right) \\
                &\leq \frac{C}{\sqrt{\log N}} N^{-2 \lambda^1} N^{-4 \frac{(\varepsilon - \eta_{\varepsilon})}{\bar{\sigma}_1}}\ .
            \end{align*}
            After multiplying by $N^{2\lambda^1}$, we conclude that the $l=1$ term in \eqref{eqn: H split.1} goes to $0$ like $N^{-c(\varepsilon)}$.
            We now show a similar estimate for a fixed $l\in \{2,...,m\}$.
            To simplify the notation, denote $(X_v^1,...,X_v^l) \circeq (\psi_v(\lambda^1),...,\psi_v(\lambda^l))$. By conditioning on the value of the vector $\boldsymbol{X} \circeq (X_v^1,...,X_v^{l-1})$, the probability in \eqref{eqn: H split.1} is equal to
            \begin{equation*}
                \int_{-\infty}^{L_N^{\star,1(\varepsilon + \eta_{\varepsilon})}(\lambda^1)} \hspace{-5mm}... \int_{-\infty}^{L_N^{\star,(l-1)(\varepsilon + \eta_{\varepsilon})}(\lambda^{l-1})} \hspace{-0.5mm}\prob{}{X_v^l \geq L_N^{\star,l\varepsilon - \eta_{\varepsilon}}(\lambda^l) \nvert \boldsymbol{X} = \boldsymbol{x}} \hspace{0.5mm}f_v(\boldsymbol{x}) ~d \boldsymbol{x}
            \end{equation*}
            where $f_v$ is the density function of $\boldsymbol{X}$. By independence of the increments, the last integral is equal to
            \begin{equation}\label{eq:lem:IGFF.order1.upper.bound.integral.1}
                \int_{-\infty}^{L_N^{\star,1(\varepsilon + \eta_{\varepsilon})}(\lambda^1)} \hspace{-5mm}... \int_{-\infty}^{L_N^{\star,(l-1)(\varepsilon + \eta_{\varepsilon})}(\lambda^{l-1})} \hspace{-1mm}\prob{}{\nabla X_v^l \geq L_N^{\star,l\varepsilon - \eta_{\varepsilon}}(\lambda^l) - x_{l-1}} \hspace{0.5mm}f_v(\boldsymbol{x}) ~d \boldsymbol{x}\ .
            \end{equation}
            Since $l \varepsilon - \eta_{\varepsilon} = (\varepsilon - l \eta_{\varepsilon}) + (l-1)(\varepsilon + \eta_{\varepsilon})$, a Gaussian estimate and the bound $\max_{v\in V_N} \var{}{\nabla X_v^l} \leq \bar\sigma_l^2 \nabla \lambda^l \log N + C$ from \eqref{eq:borne.variance.upper} give
            \begin{align}\label{eq:lem:IGFF.order1.upper.bound.gaussian.estimate}
                &\prob{}{\nabla X_v^l \geq L_N^{\star,l\varepsilon - \eta_{\varepsilon}}(\lambda^l) - x_{l-1}} \notag \\
                &\quad\leq \frac{\sqrt{\var{}{\nabla X_v^l}}}{L_N^{\star,l\varepsilon - \eta_{\varepsilon}}(\lambda^l) - x_{l-1}} \exp\left(\frac{-(\nabla L_N^{\star}(\lambda^l) + L_N^{\star,l \varepsilon - \eta_{\varepsilon}}(\lambda^{l-1}) -x_{l-1})^2}{2\var{}{\nabla X_v^l}}\right) \notag \\
                &\quad\leq \frac{C}{\sqrt{\log N}} N^{-2 \nabla \lambda^l} \exp\left(-2 \frac{(L_N^{\star,l \varepsilon - \eta_{\varepsilon}}(\lambda^{l-1}) - x_{l-1})}{\bar{\sigma}_l}\right) \notag \\
                &\quad= \frac{C}{\sqrt{\log N}} N^{-2 \nabla \lambda^l} N^{-4\frac{\varepsilon - l \eta_{\varepsilon}}{\bar{\sigma}_l}} \exp\left(-2 \frac{(L_N^{\star,(l-1)(\varepsilon + \eta_{\varepsilon})}(\lambda^{l-1}) - x_{l-1})}{\bar{\sigma}_l}\right).
            \end{align}
            To get the second inequality, we bounded the ratio using
            \begin{equation*}
                L_N^{\star,l\varepsilon - \eta_{\varepsilon}}(\lambda^l) - x_{l-1} \geq \nabla L_N^{\star}(\lambda^l) = \bar{\sigma}_l \nabla \lambda^l \log N^2
            \end{equation*}
            from the integration limits of $x_{l-1}$ in \eqref{eq:lem:IGFF.order1.upper.bound.integral.1}.
            It is convenient to do the change of variables $Y_{v,j} \circeq (\varepsilon + \eta_{\varepsilon}) \log N^2 + \nabla L_N^{\star}(\lambda^j) - \nabla X_v^j$  for all $j\in \{1,...,l-1\}$.
            Equation \eqref{eq:lem:IGFF.order1.upper.bound.integral.1} is then bounded, using \eqref{eq:lem:IGFF.order1.upper.bound.gaussian.estimate}, by
            \vspace{-2mm}
            \begin{equation}\label{eqn:  multiple}
                \frac{C N^{-4\frac{\varepsilon - l \eta_{\varepsilon}}{\bar{\sigma}_l}}}{N^{2\lambda^l} \sqrt{\log N}} \  N^{2 \lambda^{l-1}} \hspace{-1mm}\int_0^{\infty} \hspace{-2mm}\int_{-y_1}^{\infty} \hspace{-2mm}... \hspace{-1mm}\int_{-\sum_{j=1}^{l-2} y_j}^{\infty}\prod_{j=1}^{l-1} e^{-2 \frac{y_j}{\bar{\sigma}_l}}\ \frac{e^{-\frac{\left((y_j - (\varepsilon + \eta_{\varepsilon}) \log N^2) - \nabla L_N^{\star}(\lambda^j)\right)^2}{2 \var{}{Y_{v,j}}}}}{\sqrt{2\pi \var{}{Y_{v,j}}}} d \boldsymbol{y}.
            \end{equation}
            After \hspace{-0.2mm}multiplying \hspace{-0.2mm}by \hspace{-0.4mm}$N^{2\lambda^l}$\hspace{-1mm}, \hspace{-0.2mm}the \hspace{-0.2mm}$l$-th \hspace{-0.2mm}term \hspace{-0.2mm}of \hspace{-0.2mm}the \hspace{-0.2mm}sum \hspace{-0.2mm}in \hspace{-0.4mm}\eqref{eqn: H split.1}\hspace{-0.4mm} has the right decay if
            we show \hspace{-0.2mm}that \hspace{-0.2mm}the \hspace{-0.2mm}integral \hspace{-0.2mm}in \hspace{-0.4mm}\eqref{eqn:  multiple}\hspace{-0.4mm} is \hspace{-0.2mm}bounded \hspace{-0.2mm}by \hspace{-0.4mm}$\tilde{C} N^{-2\lambda^{l-1}}$\hspace{-1mm}.
            \hspace{-0.2mm}From \hspace{-0.2mm}\eqref{eq:borne.variance.upper}, \hspace{-0.2mm}we \hspace{-0.2mm}have
            \begin{equation*}
                0 < \var{}{Y_{v,j}}\leq \bar{\sigma}_j^2 \nabla \lambda^j \log N + C
            \end{equation*}
            for all $v\in V_N^o$. If the variances were all equal to $\bar{\sigma}_j^2 \nabla \lambda^j \log N + C$, the argument would be simpler.
            Extra work is needed to take care of the boundary effect of the GFF.
            We gather the result into a lemma for later use in the proof of Lemma \ref{lem:IGFF.high.points.upper.bound}.
        \end{proof}

        \begin{lemma}\label{lem: recursion}
            \hspace{-0.4mm}Let \hspace{0.2mm}$2 \hspace{-0.2mm}\leq l \hspace{-0.2mm}\leq m$ \hspace{-0.1mm}and \hspace{0.2mm}$\boldsymbol z \hspace{-0.5mm}\circeq \hspace{-0.5mm}(z_j)_{j=1}^{l-1}$ \hspace{-0.4mm}be such that $0 \hspace{-0.3mm}< \hspace{-0.3mm}z_j \hspace{-0.3mm}\leq \bar{\sigma}_j^2 \nabla \lambda^j \hspace{-0.3mm}\log N + \hspace{0.2mm}C$.
            For all $\tilde{\varepsilon} > 0$, consider the integral
            \begin{align*}
                I_{\tilde{\varepsilon}}(\boldsymbol{z})
                &\circeq \int_0^{\infty} g_1(y_1) \ ... \int_{-\sum_{j=1}^{l-3} y_j}^{\infty} \hspace{-1mm} g_{l-2}(y_{l-2}) \int_{-\sum_{j=1}^{l-2} y_j}^{\infty} \hspace{-2mm}e^{-2 a_l\sum_{j=1}^{l-1} y_j} g_{l-1}(y_{l-1}) ~d\boldsymbol{y}
            \end{align*}
            where $a_l> 1/\bar{\sigma}_{l-1}$ and
            \begin{equation*}
                g_j(y) \circeq \frac{1}{\sqrt{2\pi z_j}} \exp\left(-\frac{1}{2 z_j} \left((y - \tilde{\varepsilon} \log N^2) - \nabla L_N^{\star}(\lambda^j)\right)^2\right), \quad 1 \leq j \leq l-1\ .
            \end{equation*}
            Then $I_{\tilde{\varepsilon}}(\boldsymbol z) \leq \tilde{C} N^{-2\lambda^{l-1}}$.
        \end{lemma}

        \begin{proof}
            Let $\beta_j \circeq \frac{\nabla L_N^{\star}(\lambda^j)}{2 z_j}, ~1 \leq j \leq l-1$.
            When $a_l - \beta_{l-1} \geq 1 / \sqrt{z_{l-1}}$, the first integral with respect to $y_{l-1}$ in $I_{\tilde{\varepsilon}}(\boldsymbol{z})$ is equal to
            \begin{align*}
                & e^{-2 a_l \sum_{j=1}^{l-2} y_j} \int_{-\sum_{j=1}^{l-2} y_j}^{\infty} e^{-2 a_l y_{l-1}} \frac{1}{\sqrt{2\pi z_{l-1}}} e^{-\frac{1}{2 z_{l-1}} \left((y_{l-1} - \tilde{\varepsilon} \log N^2) - \nabla L_N^{\star}(\lambda^{l-1})\right)^2} dy_{l-1} \\
                &~~~~~\leq e^{-2 a_l \sum_{j=1}^{l-2} y_j} \frac{1}{\sqrt{z_{l-1}}} e^{-\frac{1}{2 z_{l-1}}(\nabla L_N^{\star}(\lambda^{l-1}))^2} \int_{-\sum_{j=1}^{l-2} y_j}^{\infty} e^{-2 (a_l - \beta_{l-1}) y_{l-1}} dy_{l-1} \\
                &~~~~~= e^{-2 a_l \sum_{j=1}^{l-2} y_j} \frac{1}{\sqrt{z_{l-1}}} e^{-\frac{1}{2 z_{l-1}}(\nabla L_N^{\star}(\lambda^{l-1}))^2} \frac{1}{2 (a_l - \beta_{l-1})} e^{2 (a_l - \beta_{l-1}) \sum_{j=1}^{l-2} y_j}.
            \end{align*}
            Since $z_{l-1} \leq \bar{\sigma}_{l-1}^2 \nabla \lambda^{l-1} \log N + C$ and $\nabla L_N^{\star}(\lambda^{l-1})=\bar{\sigma}_{l-1}\nabla \lambda^{l-1}\log N^2$, the above is smaller than
	        \begin{equation}\label{eqn: bound 1}
                C e^{-2 \beta_{l-1} \sum_{j=1}^{l-2} y_j} N^{-2 \nabla \lambda^{l-1}}.
	        \end{equation}
            When $a_l - \beta_{l-1} < 1 / \sqrt{z_{l-1}}$, we have by completing the square :
            \begin{align*}
                &e^{-2 a_l \sum_{j=1}^{l-2} y_j} \int_{-\sum_{j=1}^{l-2} y_j}^{\infty} \frac{e^{-2 a_l y_{l-1}}}{\sqrt{2\pi z_{l-1}}} \ e^{-\frac{\left((y_{l-1} - \tilde{\varepsilon} \log N^2) - \nabla L_N^{\star}(\lambda^{l-1})\right)^2}{2 z_{l-1}}} dy_{l-1} \\
                &~~\leq e^{-2 a_l \sum_{j=1}^{l-2} y_j} \int_{-\sum_{j=1}^{l-2} y_j}^{\infty} \frac{e^{-2 a_l (y_{l-1} - \tilde{\varepsilon} \log N^2)}}{\sqrt{2\pi z_{l-1}}} \ e^{-\frac{\left((y_{l-1} - \tilde{\varepsilon} \log N^2) - \nabla L_N^{\star}(\lambda^{l-1})\right)^2}{2 z_{l-1}}} dy_{l-1} \\
                &~~= e^{-2 a_l \sum_{j=1}^{l-2} y_j} e^{-2 a_l \nabla L_N^{\star}(\lambda^{l-1})} e^{2 a_l^2 z_{l-1}} \\
                &\quad\quad\quad\quad\quad\quad\quad \cdot \int_{-\sum_{j=1}^{l-2} y_j}^{\infty} \frac{1}{\sqrt{2\pi z_{l-1}}} e^{-\frac{\left((y_{l-1} - \tilde{\varepsilon} \log N^2) - \left(\nabla L_N^{\star}(\lambda^{l-1}) - 2 a_l z_{l-1}\right)\right)^2}{2 z_{l-1}}} dy_{l-1} \\
                &~~\leq \exp\Big(-2 a_l \sum_{j=1}^{l-2} y_j -2 a_l \nabla L_N^{\star}(\lambda^{l-1})+ 2 a_l^2 z_{l-1}\Big)\ .
            \end{align*}
            In the regime $a_l - \beta_{l-1} < 1 / \sqrt{z_{l-1}}$, note that $2 a_l^2 z_{l-1} < a_l \nabla L_N^{\star}(\lambda^{l-1}) + 2 a_l \sqrt{z_{l-1}}$. Since $z_{l-1} \leq \bar{\sigma}_{l-1}^2 \nabla \lambda^{l-1} \log N + C$, the above is smaller than
            \begin{equation*}
                \exp\Big(-2 a_l \sum_{j=1}^{l-2} y_j - a_l \nabla L_N^{\star}(\lambda^{l-1})+C \sqrt{\log N}\Big)\ .
            \end{equation*}
            By assumption, $a_l > 1/\bar{\sigma}_{l-1}$. Therefore, the above is smaller than
            \begin{equation}\label{eqn: bound 2}
                e^{-2 a_l \sum_{j=1}^{l-2} y_j} N^{-2 \nabla \lambda^{l-1}}\ .
            \end{equation}
            The second integral with respect to $y_{l-2}$ in $I_{\tilde{\varepsilon}}(\boldsymbol z)$ is evaluated similarly using \eqref{eqn: bound 1} and \eqref{eqn: bound 2} by taking
            \begin{equation*}
                a_{l-1} \circeq \min\{a_l,\beta_{l-1}\}
            \end{equation*}
            and considering whether $a_{l-1} - \beta_{l-2} \geq 1 / \sqrt{z_{l-2}}$ or not.
            Note that $a_{l-1} > 1 / \bar{\sigma}_{l-2}$ holds since $a_l > 1 / \bar\sigma_{l-1} > 1 / \bar\sigma_{l-2}$ (because the steps of $\bar{\sigma}$ are decreasing in height) and $\beta_{l-1}\geq 1 / \bar{\sigma}_{l-1} - \hspace{0.2mm}O((\log N)^{-1}) > 1 / \bar{\sigma}_{l-2}$ from the bounds on $z_{l-1}$.
            This recursive reasoning shows that $I_{\tilde{\varepsilon}}(\boldsymbol z)$ is smaller than $\tilde{C} N^{-2\lambda^{l-1}}$.
        \end{proof}

        \begin{lemma}[Lower \hspace{-0.3mm}bound on \hspace{-0.3mm}the \hspace{-0.3mm}first \hspace{-0.3mm}order \hspace{-0.3mm}of \hspace{-0.3mm}the \hspace{-0.3mm}maximum]
        \label{lem:IGFF.order1.lower.bound}
            \hspace{-1mm}Let $\{\psi_v\}_{v\in V_N}$ \hspace{-1mm}be the $(\boldsymbol{\sigma},\boldsymbol{\lambda})$-GFF on $V_N$ of Definition \ref{def:IGFF} and $\gamma^{\star}$ as in Theorem \ref{thm:IGFF.order1}.
            For all $0 < \varepsilon  < 1$, there exists a constant $c = c(\varepsilon,\boldsymbol{\sigma},\boldsymbol{\lambda}) > 0$ such that
            \begin{equation}\label{eq:IGFF.order1.lower.bound}
                \prob{}{\max_{v\in V_N} \psi_v \leq (1-\varepsilon) \gamma^{\star} \log N^2} \leq N^{-c}
            \end{equation}
            for $N$ large enough.
        \end{lemma}

        Without loss of generality, we can assume that $\lambda_i\in \Q$ for all $i\in \{0,...,M\}$.
        To see this, define $\tilde{\lambda}_i \circeq \lambda_i + \eta_i$ where $0 < \eta_i < \min_i \nabla \lambda_i$ and such that $\tilde{\lambda}_i\in \Q$ for all $i\in \{1,...,M-1\}$. Now, define a new scale-inhomogeneous Gaussian free field :
        \begin{equation*}
            \tilde{\psi}_v \circeq \sum_{i=1}^M \sigma_i \nabla \phi_v(\tilde{\lambda}_i) = \psi_v + \sum_{i=1}^{M-1} (\sigma_i - \sigma_{i+1}) (\phi_v(\tilde{\lambda}_i) - \phi_v(\lambda_i))\ .
        \end{equation*}
        As a particular case of Lemma \ref{lem:tech.lemma.2}, note that
        \begin{equation*}
            \max_{v\in V_N} \var{}{\phi_v(\tilde{\lambda}_i) - \phi_v(\lambda_i)} \leq (\tilde{\lambda}_i - \lambda_i) \log N + C = \eta_i \log N + C\ .
        \end{equation*}
        If we can show Lemma \ref{lem:IGFF.order1.lower.bound} when the $\lambda_i$'s are rational numbers, then a union bound and a Gaussian estimate yield
        \begin{align*}
            &\prob{}{\max_{v\in V_N} \psi_v \leq (1-2\varepsilon) \gamma^{\star} \log N^2}
            \leq \prob{}{\max_{v\in V_N} \tilde{\psi}_v \leq (1-\varepsilon) \gamma^{\star} \log N^2} \\
            &\quad\quad\quad\quad\quad+ \sum_{v\in V_N^o} \sum_{i=1}^{M-1} \prob{}{|\sigma_i - \sigma_{i+1}| \left|\phi_v(\tilde{\lambda}_i) - \phi_v(\lambda_i)\right| \geq (\varepsilon/(M-1)) \gamma^{\star} \log N^2} \\
            &\quad\leq N^{-c(\varepsilon,\boldsymbol{\sigma},\boldsymbol{\lambda})} + N^2 (M-1) \exp\Big(-\frac{\big((\varepsilon/(M-1)) \gamma^{\star} \log N^2\big)^2}{2 \max_i |\sigma_i - \sigma_{i+1}|^2 (\eta_i \log N + C)}\Big)\ .
        \end{align*}
        The second term can be made $O(N^{-\tilde{c}(\varepsilon,\boldsymbol{\sigma},\boldsymbol{\lambda})})$ where $\tilde{c} > 0$ is arbitrarily large, by choosing the $\eta_i$'s small enough with respect to $\varepsilon$.

        The proof of Lemma \ref{lem:IGFF.order1.lower.bound} is based on a coarse-graining of the scales introduced in \citet{MR3380419}.
        Consider $\alpha_k \circeq \frac{k}{K}, ~ 0\leq k \leq K$.
        The parameter $K\in \N$ will be chosen large enough depending on $\varepsilon$ during the proof.
        By the argument above, we can assume that $\lambda_i K\in \N_0$ for all $i\in \{0,...,M\}$, so that the $\alpha_k$'s form a finer partition of $[0,1]$ than the $\lambda_i$'s.
        The bounds in \eqref{eq:borne.variance.delta} imply that for all $k\in \{1,..., K\}$ and for all $v\in V_N^\delta$ :
        \begin{equation}\label{eqn: variance increments}
            |\var{}{\nabla\psi_v(\alpha_k)} - \sigma^2(\alpha_k)\nabla \alpha_k  \log N| \leq C(\delta)\ .
        \end{equation}
        The parameter $\delta\in (0,1/2)$ remains fixed to an arbitrary value in the rest of this section. For all $0 < \varepsilon < 1$, denote by $L_{N,\varepsilon}^{\star}$ the following sub-optimal path :
        \begin{equation*}
            L_{N,\varepsilon}^{\star}(s) = (1 - \varepsilon) L_N^{\star}(s) = (1-\varepsilon) \mathcal{J}_{\sigma^2 / \bar{\sigma}}(s) \log N^2, \ \ \ s\in [0,1]\ .
        \end{equation*}

        The proof relies on the Paley-Zygmund inequality (see Lemma \ref{lem:paley.zygmund}) applied to a {\it modified number of exceedances}.
        In fact, we consider only points in $V_N^\delta$ whose increments are almost optimal.
        Moreover, and crucially, we drop the first $r$ increments.
        We will choose $r$ during the proof.
        This allows more independence between the variables of the field, which is needed to find a tight lower bound using the Paley-Zygmund inequality.
        More precisely, define
        \begin{equation*}
             \mathcal{N}_\varepsilon^\star \circeq \sum_{v\in V_N^\delta} 1_{A_v} \ \ \text{where}\ \ A_v \circeq \{\nabla \psi_v(\alpha_j) \geq \nabla L_{N,\varepsilon}^{\star}(\alpha_j) \ \ \forall j\in \{r+1,...,K\}\}\ .
        \end{equation*}
        For a fixed $\varepsilon > 0$, there is the following inequality for $c = c(\varepsilon) > 0$ :
        \begin{equation}\label{eq:sufficient.lower.max}
            \prob{}{\max_{v\in V_N} \psi_v\geq (1 - 3\varepsilon)\gamma^\star \log N^2}
            \geq \prob{}{\mathcal{N}_\varepsilon^\star\geq 1} - O(N^{-c})\ .
        \end{equation}
        Indeed, on the event $\{\mathcal{N}_\varepsilon^\star\geq 1\}$, we have
        \begin{equation*}
            \begin{aligned}
                \max_{v\in V_N^\delta} \psi_v - \psi_v(\alpha_r)
                &\geq (1-\varepsilon) \mathcal{J}_{\sigma^2 / \bar{\sigma}}(\alpha_r,1) \log N^2 \\
                &= \vspace{-10mm}(1-\varepsilon)\gamma^\star\log N^2- (1-\varepsilon) \mathcal{J}_{\sigma^2 / \bar{\sigma}}(\alpha_r) \log N^2 \\
                &\geq (1 - 2\varepsilon)\gamma^\star\log N^2
            \end{aligned}
        \end{equation*}
        where we take $K$ large enough that $(1 - \varepsilon) \mathcal{J}_{\sigma^2 / \bar{\sigma}}(\alpha_r) < \varepsilon \gamma^{\star}$.
        Furthermore, the probability $\mathbb P(\max_{v\in V^\delta_N} \psi_v - \psi_v(\alpha_r) \geq (1 - 2\varepsilon) \gamma^\star \log N^2)$ is equal to
        \begin{equation*}
            \begin{aligned}
                &\prob{}{\max_{v\in V_N^{\delta}} \psi_v-\psi_v(\alpha_r)\geq (1 - 2\varepsilon)\gamma^\star \log N^2, \min_{v\in V_N^\delta} \psi_v(\alpha_r) > -\varepsilon \gamma^\star \log N^2} \\
                &~+ \prob{}{\max_{v\in V_N^{\delta}} \psi_v-\psi_v(\alpha_r)\geq(1 - 2\varepsilon)\gamma^\star \log N^2, \min_{v\in V_N^\delta} \psi_v(\alpha_r) \leq -\varepsilon \gamma^\star \log N^2} .
            \end{aligned}
        \end{equation*}
        The distribution of $\psi_v(\alpha_r)$ is symmetric, so the second term is smaller than
        \begin{equation}\label{eq:order1.lower.bound.gaussian.before.proof}
            \mathbb{P}\left(\max_{v\in V_N^\delta} \psi_v(\alpha_r) \geq \varepsilon \gamma^\star \log N^2\right)
            \leq N^2 \exp\left( -\frac{(\varepsilon \gamma^{\star})^2 \log N^2}{\max_i \sigma_i^2 \alpha_r}\right)
        \end{equation}
        where we used a union bound, a Gaussian estimate and \eqref{eq:borne.variance.upper} to get the inequality. This is $O(N^{-c})$ by choosing $K$ large enough for a fixed $\varepsilon$ and $r$. On the other hand, the first term is smaller than $\mathbb P(\max_{v\in V_N^\delta} \psi_v \geq (1 - 3\varepsilon)\gamma^\star \log N^2)$.
        Since $V_N \supseteq V_N^\delta$, this implies \eqref{eq:sufficient.lower.max} as claimed.

        \begin{proof}[Proof of Lemma \ref{lem:IGFF.order1.lower.bound}]
            In view of \eqref{eq:sufficient.lower.max}, it suffices to show $\prob{}{\mathcal{N}_\varepsilon^\star\geq 1} = 1 - O(N^{-c})$.
            The Paley-Zygmund inequality implies
            \begin{equation*}
                \prob{}{\mathcal{N}_\varepsilon^\star \geq 1}
                \geq \frac{(\esp{}{\mathcal{N}_\varepsilon^\star})^2}{\esp{}{(\mathcal{N}_\varepsilon^\star)^2}}\ .
            \end{equation*}
            We show
            \begin{equation}\label{eqn: lower max to show}
                \esp{}{(\mathcal{N}_\varepsilon^\star)^2} \leq (1 + O(N^{-\frac{1}{2K} (1 - (1 - \varepsilon)^2})) \ (\esp{}{\mathcal{N}_\varepsilon^\star})^2,
            \end{equation}
            which proves the claim.

            \vspace{5mm}
            The first moment is easily evaluated by the independence of the increments :
            \begin{equation*}
                \esp{}{\mathcal{N}_\varepsilon^\star} = \sum_{v\in V_N^\delta}\PP(A_v)
                =\sum_{v\in V_N^\delta}\prod_{j=r+1}^K \prob{}{\nabla \psi_v(\alpha_j)\geq \nabla L_{N,\varepsilon}^{\star}(\alpha_j)}.
            \end{equation*}
            Using Gaussian estimates and the variance estimates in \eqref{eqn: variance increments}, the probabilities are for every $j$ and $v\in V_N^\delta$ :
            \begin{align}\label{eqn: pj}
                p_{v,j}
                &\circeq \prob{}{\nabla \psi_v(\alpha_j)\geq \nabla L_{N,\varepsilon}^{\star}(\alpha_j)} \notag \\
                &\asymp \frac{1}{\sqrt{\log N}} \exp\left(-(1-\varepsilon)^2 \frac{\sigma^2(\alpha_j)}{\bar\sigma^2(\alpha_j)}\nabla\alpha_j \log N^2\right).
            \end{align}
            Write $e_j$ for the exponential term on the right-hand side of \eqref{eqn: pj}.
            The first moment satisfies
            \vspace{-1mm}
            \begin{equation}\label{eqn: first moment}
                \esp{}{\mathcal{N}_\varepsilon^\star}=\sum_{v\in V_N^{\delta}} \prob{}{A_v} \geq \frac{c(\varepsilon,\delta)}{(\log N)^{\frac{1}{2} (K - r)}} \times |V_N^\delta| \times \prod_{j=r+1}^K e_j\ .
            \end{equation}
            Now, we compare this with the second moment :
            \begin{equation*}
                \esp{}{(\mathcal{N}_\varepsilon^\star)^2} = \sum_{v,v'\in V_N^\delta}\PP(A_v\cap A_{v'})\ .
            \end{equation*}
            We divide the sum depending on the correlations between $\psi_v$ and $\psi_{v'}$.
            More precisely, recall the definition of the branching scale in \eqref{eq:rho} :
            \begin{equation*}
                \rho(v,v') \circeq \max\{\lambda\in [0,1] : [v]_\lambda \cap [v']_\lambda \neq \emptyset\}, \quad v,v'\in V_N\ .
            \end{equation*}
            Write the second moment as
            \begin{equation}\label{eqn: sum split}
                \sum_{\substack{v,v'\in V_N^{\delta} \\ \rho(v,v') < \alpha_r}} \hspace{-3mm} \prob{}{A_v\cap A_{v'}}
                ~~+ \hspace{0mm} \sum_{k=r+1}^{K-1} \hspace{-5mm} \sum_{\substack{v,v'\in V_N^{\delta} \\ \alpha_{k-1} \leq \rho(v,v') < \alpha_k}} \hspace{-7.5mm} \prob{}{A_v\cap A_{v'}}
                ~~+ \hspace{-5mm} \sum_{\substack{v,v'\in V_N^{\delta} \\ \rho(v,v') \geq \alpha_{K-1}}} \hspace{-5mm} \prob{}{A_v\cap A_{v'}}.
            \end{equation}
            In particular, the first term in \eqref{eqn: sum split} is equal to
            \begin{equation}\label{eq:lem:order1.lower.bound.first.term.bound}
                \sum_{\substack{v,v'\in V_N^{\delta} \\ \rho(v,v') < \alpha_r}} \hspace{-2mm}\prob{}{A_v} \prob{}{A_{v'}} \leq \hspace{-1mm}\sum_{v,v'\in V_N^{\delta}} \prob{}{A_v} \prob{}{A_{v'}} = (\esp{}{\mathcal{N}_\varepsilon^\star})^2.
            \end{equation}

            It remains to show that the second and third term in \eqref{eqn: sum split} are negligible compared to $(\esp{}{\mathcal{N}_\varepsilon^\star})^2$.
            We write the details for the second term since the last term is done similarly and is easier.
            By Lemma \ref{lem:psi.Markov} (following the Markov property of the GFF), note that if $\alpha_{k-1}\leq \rho(v,v')< \alpha_k$ for some $k\geq r+1$, then $\nabla \psi_{v'}(\alpha_{j'})$, $j'\geq k+1$, is independent of $\nabla \psi_{v}(\alpha_{j})$ for $j\leq k-2$ and $j\geq k+1$.
            Therefore, for $v,v'\in V_N^{\delta}$ such that $\alpha_{k-1}\leq \rho(v,v')< \alpha_k$, we have
            \begin{equation*}
                \prob{}{A_v\cap A_{v'}} \leq  \prod_{j=r+1}^{k-2} p_{v,j} \prod_{j=k+1}^{K} p_{v,j} p_{v',j} \leq \left(\prod_{j=r+1}^K e_j^2\right) \left(\frac{\prod_{j=1}^r e_j}{e_{k-1} e_k^2}\right) \left(\prod_{j=1}^{k-1} e_j\right)^{-1}
            \end{equation*}
            where we dropped the conditions on $j\in \{k-1,k\}$ for $v$ as well as the conditions on $j\leq k$ for $v'$ in the first inequality.
            We simply rearranged the probabilities and eliminated the log terms to get the last inequality. The number of pairs $v,v'\in V_N^{\delta}$ such that $\alpha_{k-1}\leq \rho(v,v') < \alpha_k$ is at most $|V_N^\delta| \times N^{2(1-\alpha_{k-1})}$. Therefore, by \eqref{eqn: first moment},
            \begin{equation}\label{eq:sum.split.second.term}
                \begin{aligned}
                    \sum_{\substack{v,v'\in V_N^{\delta} \\ \alpha_{k-1} \leq \rho(v,v') < \alpha_k}} \hspace{-7mm} \prob{}{A_v\cap A_{v'}}
                    &\leq \frac{(\esp{}{\mathcal{N}_\varepsilon^\star})^2}{(\log N)^{-K}} \times N^{-2\alpha_{k-1} (1 - (1 - \varepsilon)^2)} \left(\frac{\prod_{j=1}^r e_j}{e_{k-1} e_k^2}\right) \\
                    &\phantom{\times \frac{N^{-2\alpha_{k-1} (1 - \varepsilon)^2}}{\prod_{j=1}^{k-1} e_j}}
                \end{aligned}
            \end{equation}
            \vspace{-18mm}
            \begin{equation*}
                \quad\quad\quad\quad\quad\quad\quad\quad\quad\quad\times \frac{N^{-2\alpha_{k-1} (1 - \varepsilon)^2}}{\prod_{j=1}^{k-1} e_j}
            \end{equation*}
            The \hspace{-0.2mm}right-hand \hspace{-0.2mm}side \hspace{-0.2mm}of \hspace{-0.6mm}\eqref{eq:sum.split.second.term}\hspace{-0.6mm} is \hspace{-0.2mm}separated \hspace{-0.2mm}in \hspace{-0.2mm}three \hspace{-0.2mm}factors \hspace{-0.2mm}by \hspace{-0.4mm}$\times$.
            \hspace{-1mm}The \hspace{-0.2mm}third \hspace{-0.2mm}factor is bounded by $1$ because
            \begin{equation*}
                \int_{0}^t \frac{\sigma^2(s)}{\bar\sigma^2(s)}ds \leq t, \ \ \ t\in (0,1],\vspace{1mm}
            \end{equation*}
            by definition of $\bar{\sigma}$.
            To bound the second factor, set $r \geq 3$ independently of any other variable.
            Note that if $r$ depended on $K$, the bound in \eqref{eq:order1.lower.bound.gaussian.before.proof} would not necessarily tend to $0$.
            There are two cases to consider : $\alpha_k \leq \lambda_1$ and $\alpha_k > \lambda_1$.
            When $\alpha_k \leq \lambda_1$, the ratio of exponentials is bounded by $1$ because $e_1 e_2 e_3 = e_{k-1} e_k^2$ and we have $N^{-2\alpha_{k-1} (1 - (1 - \varepsilon)^2)} \leq N^{-\frac{1}{K} (1 - (1 - \varepsilon)^2)}$ since $\alpha_{k-1} \geq \alpha_r \geq 1 / (2K)$.
            When $\alpha_k > \lambda_1$, the ratio of exponentials is bounded by $N^{\lambda_1 (1 - (1 - \varepsilon)^2)}$ by choosing $K$ large enough for a fixed $\varepsilon$ and we have $N^{-2\alpha_{k-1} (1 - (1 - \varepsilon)^2)} \leq N^{-2\lambda_1 (1 - (1 - \varepsilon)^2)}$ because $\alpha_{k-1} \geq \lambda_1$. Since $\lambda_1 \geq 1/K$, the right-hand side of \eqref{eq:sum.split.second.term} is always bounded by
            \begin{equation*}
                (\esp{}{\mathcal{N}_\varepsilon^\star})^2 \ (\log N)^K \times N^{-\frac{1}{K} (1 - (1 - \varepsilon)^2)} \notag\ .
            \end{equation*}
            With \eqref{eq:lem:order1.lower.bound.first.term.bound}, this shows \eqref{eqn: lower max to show} and concludes the proof of the lemma.
        \end{proof}

    \subsection{Log-number of high points}

        The proof of the upper bound for the log-number of high-points uses an argument based on the path at every scale $\lambda^l$ similar to the one in Lemma \ref{lem:IGFF.order1.upper.bound}.
        Recall the definition of the critical levels $\gamma^l$ and the entropy $\mathcal E_\gamma$ in Theorem \ref{thm:IGFF.high.points}.

        \begin{lemma}[Upper bound on the log-number of high points]\label{lem:IGFF.high.points.upper.bound}
            Let $\{\psi_v\}_{v\in V_N}$ be the $(\boldsymbol{\sigma},\boldsymbol{\lambda})$-GFF on $V_N$ of Definition \ref{def:IGFF} and $\gamma^{\star}$ as defined in Theorem \ref{thm:IGFF.order1}. Also, let $\gamma^{l-1} < \gamma \leq \gamma^{l}$ for some $l\in \{1,...,m\}$.
            For all $0 < \varepsilon < (\gamma - \gamma^{l-1})/m$, there exists a constant $c = c(\gamma,\varepsilon,\boldsymbol{\sigma},\boldsymbol{\lambda}) > 0$ such that
            \begin{equation}\label{eq:IGFF.high.points.upper.bound}
                \prob{}{|\mathcal{H}_N^{\gamma}| \geq N^{2\mathcal{E}_{\gamma} + \varepsilon}} \leq N^{-c}
            \end{equation}
            for $N$ large enough.
        \end{lemma}

        \begin{proof}
            Recall the definition of the optimal path $L_N^{\gamma}$ from \eqref{eqn: gamma optimal} and the notation $R_{\lambda^j}$ for the set of representatives at scale $\lambda^j$.
            Consider
            \begin{equation*}
                \mathcal{H}_{N,j}^{\gamma,\varepsilon} \circeq \left\{v\in R_{\lambda^j} \nvert \psi_v(\lambda^j) \geq L_N^{\gamma + j\varepsilon}(\lambda^j)\right\}, \ \ \ 1 \leq j \leq m \ .
            \end{equation*}
            Since $R_{\lambda^m} = V_N$, note that
            \begin{equation*}
                \mathcal{H}_N^{\gamma} = \mathcal{H}_{N,m}^{\gamma,0} = \mathcal{H}_{N,m}^{\gamma - m\varepsilon,\varepsilon}.
            \end{equation*}
            This is useful because the hypothesis $\varepsilon < (\gamma - \gamma^{l-1})/m$ implies $\gamma^{l-1} < \gamma - j\varepsilon \leq \gamma^l$, which means (in particular) that for all $j\in \{1,...,l-1\}$,
            \begin{equation}\label{eq:lem:IGFF.high.points.upper.bound.text.max}
                \text{the paths $L_N^{\star}$ and $L_N^{\gamma - j\varepsilon}$ coincide on the interval $[0,\lambda^j]$.}
            \end{equation}
            The idea is to split the probability that at least $N^{2 \mathcal{E}_{\gamma} + \varepsilon}$ points in $V_N$ reach the optimal height by looking at the first scale $\lambda^j, ~1 \leq j \leq l-1$, where the set $\mathcal{H}_{N,j}^{\gamma - j\varepsilon,\varepsilon}$ is not empty. As for the maximum, this yields the appropriate constraints along the scales to get the correct upper bound. A union bound in \eqref{eq:IGFF.high.points.upper.bound} gives
            \begin{align}\label{eq:lem:IGFF.high.points.upper.bound.beginning.1}
                &\prob{}{|\mathcal{H}_N^{\gamma}| \geq N^{2\mathcal{E}_{\gamma} + \varepsilon}} = \prob{}{|\mathcal{H}_{N,m}^{\gamma - m\varepsilon,\varepsilon}| \geq N^{2\mathcal{E}_{\gamma} + \varepsilon}} \notag \\
                &\quad\leq \prob{}{
                    \begin{array}{l}
                        |\mathcal{H}_{N,1}^{\gamma - 1\varepsilon,\varepsilon}| = ... = |\mathcal{H}_{N,l-1}^{\gamma - (l-1)\varepsilon,\varepsilon}| = 0 \\
                        \text{and } |\mathcal{H}_{N,m}^{\gamma - m\varepsilon,\varepsilon}| \geq N^{2 \mathcal{E}_{\gamma} + \varepsilon}
                    \end{array}
                    } + \sum_{j=1}^{l-1} \prob{}{|\mathcal{H}_{N,j}^{\gamma - j\varepsilon,\varepsilon}| \geq 1}.
            \end{align}
            Because of \eqref{eq:lem:IGFF.high.points.upper.bound.text.max}, the probabilities in the sum are bounded by $N^{-c(\varepsilon)}$ in exactly the same manner as $\mathbb{P}(|\mathcal{H}_{N,m}^{\star,\varepsilon}| \geq 1)$ in Lemma \ref{lem:IGFF.order1.upper.bound}. The first probability in \eqref{eq:lem:IGFF.high.points.upper.bound.beginning.1} is bounded by
            \begin{align}\label{eq:lem:IGFF.high.points.upper.bound.beginning.2}
                &\prob{}{\left|\left\{v\in V_N \nvert
                \begin{array}{l}
                    \psi_v \geq L_N^{\gamma}(1) ~\text{and}~ \psi_{v_{\lambda^j}}(\lambda^j) < L_N^{\gamma}(\lambda^j) \\
                    \text{for all } 1 \leq j \leq l-1
                \end{array}\hspace{-1mm}
                \right\}\right| \geq N^{2 \mathcal{E}_{\gamma} + \varepsilon}} \notag \\
                &\quad\leq C e^{-c(\eta_{\varepsilon}) (\log N)^2} + N^{-\varepsilon} N^{-2 \mathcal{E}_{\gamma}} N^2 \max_{v\in V_N^o} \prob{}{
                \begin{array}{l}
                    \psi_v \geq L_N^{\gamma - \eta_{\varepsilon}}(1) ~\text{and} \\
                    \psi_v(\lambda^j) < L_N^{\gamma + j \eta_{\varepsilon}}(\lambda^j) \\
                    \text{for all } 1\leq j \leq l-1
                \end{array}
                \hspace{-1mm}}
            \end{align}
            using Markov's inequality and using the event $A_{\varepsilon}$ as in \eqref{eqn: H split.1}, where we impose
            \begin{equation*}
                0 < \eta_{\varepsilon} < \min\{\gamma,\mathcal{J}_{\sigma^2}(1) \varepsilon/(4\gamma),\bar{\sigma}_{l-1} \varepsilon / (4 l c_{\gamma}),\varepsilon / m\}
            \end{equation*}
            this time around.
            See \eqref{eq:lem:IGFF.high.points.upper.bound.constant.c.gamma} for the definition of $c_{\gamma}$. See just below and also \eqref{eq:lem:IGFF.high.points.upper.bound.asymp} for the justification of the constraints on $\eta_{\varepsilon}$. When $l = 1$, a Gaussian estimate and the bound $\max_{v\in V_N} \hspace{-1mm}\var{}{\psi_v} \hspace{-0.5mm}\leq \mathcal{J}_{\sigma^2}(1) \log N + \hspace{0.1mm}C$ from \eqref{eq:borne.variance.upper} yield
            \begin{equation*}
                \prob{}{\psi_v \geq L_N^{\gamma - \eta_{\varepsilon}}(1)} \leq \frac{\sqrt{\var{}{\psi_v}}}{L_N^{\gamma - \eta_{\varepsilon}}(1)} \exp\left(-\frac{(L_N^{\gamma - \eta_{\varepsilon}}(1))^2}{2 \var{}{\psi_v}}\right) \leq \frac{C(\gamma) N^{-2 + 2 \mathcal{E}_{\gamma}}}{\sqrt{\log N}} N^{\frac{4\gamma}{\mathcal{J}_{\sigma^2}(1)} \eta_{\varepsilon}}
            \end{equation*}
            because $L_N^{\gamma - \eta_{\varepsilon}}(1) = 2 (\gamma - \eta_{\varepsilon}) \log N$ and $\mathcal{E}_{\gamma} = 1 - \gamma^2/\mathcal{J}_{\sigma^2}(1)$ in this case.
            This proves that the second term in \eqref{eq:lem:IGFF.high.points.upper.bound.beginning.2} decays like $N^{-c(\gamma,\varepsilon)}$, as needed.

            It remains to show a similar estimate for a fixed $l\in \{2,...,m\}$. To simplify the notation, denote $(X_v^1,...,X_v^{l-1},X_v^m) \circeq (\psi_v(\lambda^1),...,\psi_v(\lambda^{l-1}),\psi_v)$. By conditioning  on the value of the vector $\boldsymbol{X} \circeq (X_v^1,...,X_v^{l-1})$, the probability in \eqref{eq:lem:IGFF.high.points.upper.bound.beginning.2} is equal to
            \begin{equation*}
                \int_{-\infty}^{L_N^{\gamma + 1 \eta_{\varepsilon}}(\lambda^1)} \hspace{-5mm}... \int_{-\infty}^{L_N^{\gamma + (l-1) \eta_{\varepsilon}}(\lambda^{l-1})} \hspace{-0.5mm}\prob{}{X_v^m \geq L_N^{\gamma - \eta_{\varepsilon}}(1) \nvert \boldsymbol{X} = \boldsymbol{x}} \hspace{0.5mm}f_v(\boldsymbol{x}) ~d \boldsymbol{x}
            \end{equation*}
            where $f_v$ is the density function of $\boldsymbol{X}$. By independence of the increments, the last integral is equal to
            \begin{equation}\label{eq:lem:IGFF.high.points.upper.bound.integral.1}
                \int_{-\infty}^{L_N^{\gamma + 1 \eta_{\varepsilon}}(\lambda^1)} \hspace{-5mm} ... \int_{-\infty}^{L_N^{\gamma + (l-1) \eta_{\varepsilon}}(\lambda^{l-1})} \hspace{-0.5mm} \prob{}{X_v^m - X_v^{l-1} \geq L_N^{\gamma - \eta_{\varepsilon}}(1) - x_{l-1}} \hspace{0.5mm}f_v(\boldsymbol{x}) ~d \boldsymbol{x}\ .
            \end{equation}
            The bound $\max_{v\in V_N} \var{}{X_v^m - X_v^{l-1}} \leq \mathcal{J}_{\sigma^2}(\lambda^{l-1},1) \log N + C$ from \eqref{eq:borne.variance.upper} and a Gaussian estimate show that
            \begin{align*}
                &\prob{}{X_v^m - X_v^{l-1} \geq L_N^{\gamma - \eta_{\varepsilon}}(1) - x_{l-1}} \\
                &= \prob{}{X_v^m - X_v^{l-1} \geq L_N^{\gamma}(1) - L_N^{\gamma}(\lambda^{l-1}) + L_N^{\gamma - \eta_{\varepsilon}}(\lambda^{l-1}) - x_{l-1}} \\
                &~~\leq \frac{C(\gamma)}{\sqrt{\log N}} N^{-2 \frac{(\gamma - \mathcal{J}_{\sigma^2 / \bar{\sigma}}(\lambda^{l-1}))^2}{\mathcal{J}_{\sigma^2}(\lambda^{l-1},1)}} \exp\left(-2\frac{(\gamma - \mathcal{J}_{\sigma^2 / \bar{\sigma}}(\lambda^{l-1}))}{\mathcal{J}_{\sigma^2}(\lambda^{l-1},1)}(L_N^{\gamma - \eta_{\varepsilon}}(\lambda^{l-1}) - x_{l-1})\right)
            \end{align*}
            where we introduced $L_N^{\gamma}(\lambda^{l-1})$ and used \eqref{eqn: gamma optimal}.
            By definition of $\mathcal E_\gamma$ and the definition of $\gamma^{l-1}$ in \eqref{eq:critic.levels}, this is equal to
            \begin{align}\label{eq:lem:IGFF.high.points.upper.bound.bound.before.constant}
                &\frac{C(\gamma) N^{-2 + 2\mathcal{E}_{\gamma}}}{\sqrt{\log N}} N^{2 \lambda^{l-1}}\exp\left(-2 \left[\frac{(\gamma - \gamma^{l-1})}{\mathcal{J}_{\sigma^2}(\lambda^{l-1},1)} + \frac{1}{\bar{\sigma}_{l-1}}\right] (L_N^{\gamma - \eta_{\varepsilon}}(\lambda^{l-1}) - x_{l-1})\right) \notag \\
                &= \frac{C(\gamma) N^{-2 + 2\mathcal{E}_{\gamma}}}{\sqrt{\log N}} N^{2 \lambda^{l-1}} N^{\frac{4 l c_{\gamma}}{\bar{\sigma}_{l-1}}\eta_{\varepsilon}}\exp\left(-2 \frac{c_{\gamma}}{\bar{\sigma}_{l-1}} (L_N^{\gamma + (l-1) \eta_{\varepsilon}}(\lambda^{l-1}) - x_{l-1})\right)
            \end{align}
            where
            \begin{equation}\label{eq:lem:IGFF.high.points.upper.bound.constant.c.gamma}
                c_{\gamma} \circeq \frac{(\gamma - \gamma^{l-1}) \bar{\sigma}_{l-1}}{\mathcal{J}_{\sigma^2}(\lambda^{l-1},1)} + 1 > 1 \ .
            \end{equation}
            Putting the bound \eqref{eq:lem:IGFF.high.points.upper.bound.bound.before.constant} in \eqref{eq:lem:IGFF.high.points.upper.bound.integral.1} and in \eqref{eq:lem:IGFF.high.points.upper.bound.beginning.2},
            we get that the first term in \eqref{eq:lem:IGFF.high.points.upper.bound.beginning.1} decays like
            \begin{equation}\label{eq:lem:IGFF.high.points.upper.bound.asymp}
                N^{-\left(\varepsilon - \frac{4 l c_{\gamma}}{\bar{\sigma}_{l-1}} \eta_{\varepsilon}\right)}
            \end{equation}
            provided that
            \vspace{-2mm}
            \begin{equation*}
                \int_0^{\infty} \int_{-y_1}^{\infty} ... \int_{-\sum_{j=1}^{l-2} y_j}^{\infty}  \prod_{j=1}^{l-1} e^{-2 \frac{c_{\gamma}}{\bar{\sigma}_{l-1}} y_j}\  \frac{e^{-\frac{\left((y_j - \eta_{\varepsilon} \log N^2) - \nabla L_N^{\gamma}(\lambda^j)\right)^2}{2 \var{}{Y_{v,j}}}}}{\sqrt{2\pi \var{}{Y_{v,j}}}} d \boldsymbol{y} \leq \tilde{C} N^{-2\lambda^{l-1}},
            \end{equation*}
            where $Y_{v,j} \circeq \eta_{\varepsilon} \log N^2 + \nabla L_N^{\gamma}(\lambda^j) - \nabla X_v^j$.
            Similarly to \eqref{eqn:  multiple}, the integral has the right decay as a consequence of Lemma \ref{lem: recursion}, with $a_l \circeq c_{\gamma} / \bar{\sigma}_{l-1} > 1 / \bar{\sigma}_{l-1}$, because $L_N^\star$ and $L_N^\gamma$ coincide on the interval $[0,\lambda^{l-1}]$.
        \end{proof}

        \begin{lemma}[Lower bound on the log-number of high points]\label{lem:IGFF.high.points.lower.bound}
            Let $\{\psi_v\}_{v\in V_N}$ be the $(\boldsymbol{\sigma},\boldsymbol{\lambda})$-GFF on $V_N$ of Definition \ref{def:IGFF} and $\gamma^{\star}$ as in Theorem \ref{thm:IGFF.order1}. Let $\gamma > 0$ be such that $\gamma^{l-1} \leq \gamma < \gamma^{l}$ for some $l\in \{1,...,m\}$.
            For all $0 < \varepsilon < \min\{1/4,(\gamma^l - \gamma)/(4\gamma)\}$, there exists a constant $c = c(\gamma,\varepsilon,\boldsymbol{\sigma},\boldsymbol{\lambda}) > 0$ such that
            \begin{equation*}
                \prob{}{|\mathcal{H}_N^{\gamma}| < N^{2\mathcal{E}_{\gamma} - \tilde\varepsilon}} \leq N^{-c}
            \end{equation*}
            for $N$ large enough, where $\tilde\varepsilon \circeq \frac{24 (\gamma^{\star})^2}{\bar{\sigma}_m^2 \nabla \lambda^m} \varepsilon$.
        \end{lemma}

        We use the same notations as in the proof of Lemma \ref{lem:IGFF.order1.lower.bound}.
        As before, we can assume, without loss of generality, that $\lambda_i K\in \N_0$ for all $\{0,...,M\}$ so that the $\alpha_k$'s form a finer partition of $[0,1]$ than the $\lambda_i$'s.
        The parameter $K\in \N$ will be chosen large enough depending on $\gamma$ and $\varepsilon$ during the proof. Again, we restrict ourselves to $V_N^{\delta}$ to ensure that for all $k\in \{1,...,K\}$ and for all $v\in V_N^{\delta}$ :
        \begin{equation}\label{eq:variance.increments.high.points}
            |\var{}{\nabla\psi_v(\alpha_k)} - \sigma^2(\alpha_k)\nabla \alpha_k  \log N| \leq C(\delta)\ .
        \end{equation}
        The parameter $\delta\in (0,1/2)$ remains fixed to an arbitrary value in the remainder of this section.
        Next, define the path :
        \begin{equation*}
            L_{N,\varepsilon}^{\gamma}(s) \circeq (1 - \varepsilon) L_N^{\gamma (1 + 4\varepsilon)}(s), \ \ \ s\in [0,1]\ .
        \end{equation*}
        Since $\varepsilon < (\gamma^l - \gamma)/(4\gamma)$ by hypothesis, we have $\gamma^{l-1} \leq \gamma < \gamma (1 + 4\varepsilon) < \gamma^l$. This condition implies that the increments of the path $L_{N,\varepsilon}^{\gamma}$ are always bounded by the increments of the sub-optimal path $L_{N,\varepsilon}^{\star}$ (see Figure \ref{fig:optimal}), namely
        \begin{equation}\label{eq:lem:IGFF.high.points.lower.bound.sub.optimal.path.inequality}
            L_{N,\varepsilon}^{\gamma}(s_2) - L_{N,\varepsilon}^{\gamma}(s_1) \leq L_{N,\varepsilon}^{\star}(s_2) - L_{N,\varepsilon}^{\star}(s_1), \ \ \ 0 \leq s_1 \leq s_2 \leq 1\hspace{0.3mm}.
        \end{equation}
        Indeed, the paths $L_N^{\gamma(1 + 4\varepsilon)}$ and $L_N^{\star}$ coincide on the interval $[0,\lambda^{l-1}]$. Moreover, when $s\in (\lambda^{l-1},1]$, we have by the definition of the critic levels $\gamma^l$ in \eqref{eq:critic.levels} and the optimal path $L_N^{\gamma(1 + 4\varepsilon)}$ in \eqref{eqn: gamma optimal} :
        \begin{align*}
            \frac{d}{ds} \frac{(L_N^{\gamma(1 + 4\varepsilon)}(s) - L_N^{\star}(s))}{\log N^2}
            &= \frac{d}{ds} \int_{\lambda^{l-1}}^s \hspace{-1mm} \left[\sigma^2(u) \frac{\left(\gamma(1 + 4\varepsilon) - \mathcal{J}_{\sigma^2/\bar{\sigma}}(\lambda^{l-1})\right)}{\mathcal{J}_{\sigma^2}(\lambda^{l-1},1)} - \frac{\sigma^2(u)}{\bar{\sigma}(u)}\right] \hspace{-1mm} du \\
            &\leq \frac{\sigma^2(s)}{\bar{\sigma}_l} - \frac{\sigma^2(s)}{\bar{\sigma}(s)} \qquad \text{since } \gamma (1 + 4\varepsilon) < \gamma^l \\
            &\leq 0 \qquad \text{since $\bar{\sigma}$ is non-increasing}.
        \end{align*}
        This proves inequality \eqref{eq:lem:IGFF.high.points.lower.bound.sub.optimal.path.inequality}. By hypothesis, we also have $\varepsilon < 1/4$, which yields
        \begin{equation}\label{eq:high.points.lower.bound.second.hypothesis}
            L_{N,\varepsilon}^{\gamma}(1) = (1 - \varepsilon) (1 + 4\varepsilon) \gamma \log N^2 > (1 + 2\varepsilon) \gamma \log N^2\ .
        \end{equation}

        The proof again relies on the Paley-Zygmund inequality applied to a \textit{modified number of exceedances} where we consider only points in $V_N^{\delta}$ whose increments are almost optimal. We drop the first $r$ increments to allow more independence which is needed for the second-moment method to work.
        We can choose $r \geq 3$ independently of any other variable as in the proof of Lemma \ref{lem:IGFF.order1.lower.bound}.
        The case $l = 1$ is easier to deal with, so we omit the details. Assume $l\in \{2,...,m\}$ and define
        \begin{equation*}
            \mathcal{N}_{\varepsilon}^{\gamma} \circeq \sum_{v\in V_N^{\delta}} 1_{A_v} \ \ \text{where}\ \ A_v \circeq \{\nabla \psi_v(\alpha_j) \geq \nabla L_{N,\varepsilon}^{\gamma}(\alpha_j) \ \ \forall j\in \{r+1,...,K\}\}.
        \end{equation*}
        Note that for a fixed $\varepsilon > 0$, there is the following inequality for $c = c(\gamma,\varepsilon) > 0$ :
        \begin{equation}\label{eq:lem:IGFF.high.points.lower.bound.before.proof}
            \prob{}{|\mathcal{H}_N^{\gamma}| \geq N^{2 \mathcal{E}_{\gamma} - \tilde\varepsilon}} \geq \prob{}{\mathcal{N}_{\varepsilon}^{\gamma} \geq N^{2 \mathcal{E}_{\gamma} - \tilde\varepsilon}} - O(N^{-c}).
        \end{equation}
        Indeed, the probability $\prob{}{\mathcal{N}_{\varepsilon}^{\gamma} \geq N^{2 \mathcal{E}_{\gamma} - \tilde\varepsilon}}$ is equal to
        \begin{equation}
            \begin{aligned}\label{eq:lem:IGFF.high.points.lower.bound.two.prob.decomp}
            &\prob{}{\mathcal{N}_{\varepsilon}^{\gamma} \geq N^{2 \mathcal{E}_{\gamma} - \tilde\varepsilon},\min_{v\in V_N^{\delta}} \psi_v(\alpha_r) > -\varepsilon \gamma \log N^2} \\
            &\quad+ \prob{}{\mathcal{N}_{\varepsilon}^{\gamma} \geq N^{2 \mathcal{E}_{\gamma} - \tilde\varepsilon},\min_{v\in V_N^{\delta}} \psi_v(\alpha_r) \leq -\varepsilon \gamma \log N^2}
            \end{aligned}
        \end{equation}
        To simplify the argument, assume from now on that $K$ is large enough to ensure $\alpha_r \leq \lambda^{l-1}$.
        The first probability in \eqref{eq:lem:IGFF.high.points.lower.bound.two.prob.decomp} is smaller than $\prob{}{|\mathcal{H}_N^{\gamma}| \geq N^{2 \mathcal{E}_{\gamma} - \tilde\varepsilon}}$ because the points $v\in V_N^{\delta}$ that are contributing to the sum $\mathcal{N}_{\varepsilon}^{\gamma}$, on the event $\{\min_{v\in V_N^{\delta}} \psi_v(\alpha_r) > -\varepsilon \gamma \log N^2\}$, are also in $\mathcal{H}_N^{\gamma}$. Indeed, when $1_{A_v} = 1$,
        \begin{align}\label{eq:lem:IGFF.high.points.lower.bound.two.prob.decomp.2}
            \psi_v - \psi_v(\alpha_r)
            &\geq (1 - \varepsilon) L_N^{\gamma(1 + 4\varepsilon)}(1) - (1 - \varepsilon) L_N^{\star}(\alpha_r) \notag \\
            &= (1 - \varepsilon) (1 + 4\varepsilon) \gamma \log N^2 - (1 - \varepsilon) \mathcal{J}_{\sigma^2 / \bar{\sigma}}(\alpha_r) \log N^2 \notag \\
            &\geq (1 + \varepsilon) \gamma \log N^2
        \end{align}
        where we take $K$ large enough that $(1 - \varepsilon) \mathcal{J}_{\sigma^2 / \bar{\sigma}}(\alpha_r) < \varepsilon \gamma$ and use \eqref{eq:high.points.lower.bound.second.hypothesis} to obtain the last inequality in \eqref{eq:lem:IGFF.high.points.lower.bound.two.prob.decomp.2}. The distribution of $\psi_v(\alpha_r)$ is symmetric, so the second probability in \eqref{eq:lem:IGFF.high.points.lower.bound.two.prob.decomp} is smaller than
        \begin{equation*}
            \prob{}{\max_{v\in V_N^{\delta}} \psi_v(\alpha_r) \geq \varepsilon \gamma \log N^2} \leq N^2 \exp\left(-\frac{(\varepsilon \gamma)^2 \log N^2}{\max_i \sigma_i^2 \alpha_r}\right)
        \end{equation*}
        where we used a union bound, a Gaussian estimate and \eqref{eq:borne.variance.upper} to get the inequality.
        This is $O(N^{-c})$ by choosing $K$ large enough for a fixed $\varepsilon$ and $r$. Therefore, we have \eqref{eq:lem:IGFF.high.points.lower.bound.before.proof} as claimed.

        \begin{proof}[Proof of Lemma \ref{lem:IGFF.high.points.lower.bound}]
            In view of \eqref{eq:lem:IGFF.high.points.lower.bound.before.proof}, it suffices to show that $\prob{}{\mathcal{N}_{\varepsilon}^{\gamma} \geq N^{2 \mathcal{E}_{\gamma} - \tilde\varepsilon}} = 1 - O(N^{-c})$.
            The Paley-Zygmund inequality (Lemma \ref{lem:paley.zygmund}) implies
            \begin{equation}\label{eqn: PZ high}
                \prob{}{\mathcal{N}_{\varepsilon}^{\gamma} \geq N^{2 \mathcal{E}_{\gamma} - \tilde\varepsilon}} \geq \left(1 - \frac{N^{2 \mathcal{E}_{\gamma} - \tilde\varepsilon}}{\esp{}{\mathcal{N}_{\varepsilon}^{\gamma}}}\right)^2 \frac{(\esp{}{\mathcal{N}_{\varepsilon}^{\gamma}})^2}{\esp{}{(\mathcal{N}_{\varepsilon}^{\gamma})^2}}.
            \end{equation}
            First, we make sure that $N^{2 \mathcal{E}_{\gamma} - \tilde\varepsilon}/\esp{}{\mathcal{N}_{\varepsilon}^{\gamma}}\to 0$ as $N\to \infty$. By independence of the increments and the variance estimate \eqref{eq:variance.increments.high.points},
            Gaussian estimates yield for some constant $c = c(\gamma,\varepsilon,\delta) > 0$ :
            \vspace{-2mm}
            \begin{align}\label{eq:lem:IGFF.high.points.lower.bound.first.moment.lower.bound}
                \esp{}{\mathcal{N}_{\varepsilon}^{\gamma}}
                &= \sum_{v\in V_N^{\delta}} \prob{}{A_v} = \sum_{v\in V_N^{\delta}} \prod_{j=r+1}^K \prob{}{\nabla \psi_v(\alpha_j) \geq \nabla L_{N,\varepsilon}^{\gamma}(\alpha_j)} \notag \\
                &\geq c \cdot (\log N)^{-\frac{1}{2}(K - r)} N^{2(1 - (1 - \varepsilon)^2) + 2(1 - \varepsilon)^2 \mathcal{E}_{\gamma(1 + 4\varepsilon)} + 2(1 - \varepsilon)^2 \int_0^{\alpha_r} \frac{\sigma^2(s)}{\bar{\sigma}^2(s)}ds} \notag \\
                &\geq N^{2(1 - (1 - \varepsilon)^2) + 2(1 - \varepsilon)^2 \mathcal{E}_{\gamma(1 + 4\varepsilon)}}.
            \end{align}
            By the definition of $\mathcal{E}_{\gamma}$ in Theorem \ref{thm:IGFF.high.points}, and because $\gamma^{l-1} \leq \gamma < \gamma (1 + 4\varepsilon) < \gamma^l$,
            \begin{align}\label{eq:high.points.lower.bound.entropy.continuity}
                \left|\mathcal{E}_{\gamma(1 + 4\varepsilon)} - \mathcal{E}_{\gamma}\right|
                &= \frac{(\gamma(1 + 4\varepsilon) - \mathcal{J}_{\sigma^2 / \bar{\sigma}}(\lambda^{l-1}))^2 - (\gamma - \mathcal{J}_{\sigma^2 / \bar{\sigma}}(\lambda^{l-1}))^2}{\mathcal{J}_{\sigma^2}(\lambda^{l-1},1)} \notag \\
                &= \frac{16 \varepsilon^2 \gamma^2 + 8 \varepsilon \gamma (\gamma - \mathcal{J}_{\sigma^2 / \bar{\sigma}}(\lambda^{l-1}))}{\mathcal{J}_{\sigma^2}(\lambda^{l-1},1)} \leq \frac{12(\gamma^{\star})^2}{\bar{\sigma}_m^2 \nabla \lambda^m} \varepsilon \circeq \tilde\varepsilon/2
            \end{align}
            where we used $\varepsilon < 1/4$, $\gamma < \gamma^{\star}$ and $\mathcal{J}_{\sigma^2}(\lambda^{l-1},1) \geq \mathcal{J}_{\sigma^2}(\lambda^{m-1},1) = \bar{\sigma}_m^2 \nabla \lambda^m$ to obtain the inequality.
            By inserting the bound \eqref{eq:high.points.lower.bound.entropy.continuity} in \eqref{eq:lem:IGFF.high.points.lower.bound.first.moment.lower.bound}, we get
            \begin{equation*}
                \esp{}{\mathcal{N}_{\varepsilon}^{\gamma}} \geq N^{2(1 - (1 - \varepsilon)^2) + 2 (1 - \varepsilon)^2 (\mathcal{E}_{\gamma} - \tilde\varepsilon/2)} = N^{2\mathcal{E}_{\gamma} - \tilde\varepsilon} N^{2(1 - (1 - \varepsilon)^2) (1 - \mathcal{E}_{\gamma} + \tilde\varepsilon/2)}.
            \end{equation*}
            Since $(1 - \varepsilon)^2 < 1$ and $\mathcal{E}_{\gamma} \leq 1$, it proves the assertion that $N^{2 \mathcal{E}_{\gamma} - \tilde\varepsilon}/\esp{}{\mathcal{N}_{\varepsilon}^{\gamma}}\to 0$ and also justify the use of the Paley-Zygmund inequality.
            In view of \eqref{eqn: PZ high}, it suffices to show, like in Lemma \ref{lem:IGFF.order1.lower.bound}, that
            \begin{equation*}
                \esp{}{(\mathcal{N}_\varepsilon^\gamma)^2} \leq (1 + O(N^{-\frac{1}{2K} (1 - (1 - \varepsilon)^2})) \ (\esp{}{\mathcal{N}_\varepsilon^\gamma})^2
            \end{equation*}
            to prove the lemma. The proof is almost identical to the proof of Lemma \ref{lem:IGFF.order1.lower.bound}. Indeed, by Gaussian estimates and the variance estimates in \eqref{eq:variance.increments.high.points}, the probabilities on the increments in $A_v$ are for every $j$ and $v\in V_N^\delta$ :
            \begin{equation*}
                \tilde p_{v,j} \circeq \prob{}{\nabla \psi_v(\alpha_j)\geq \nabla L_{N,\varepsilon}^{\gamma}(\alpha_j)} \asymp \frac{\tilde e_j}{\sqrt{\log N}}
            \end{equation*}
            where the $\tilde e_j$'s are the corresponding exponential factors. The proof is exactly the same up to \eqref{eq:sum.split.second.term} with $\tilde e_j$'s instead of $e_j$'s. From there, the third factor in the decomposition is still bounded by $1$ because of property \eqref{eq:lem:IGFF.high.points.lower.bound.sub.optimal.path.inequality}, and the rest of the argument follows if we choose $K$ large enough for a fixed $\varepsilon$ and $\gamma$. This ends the proof of the lemma.
        \end{proof}

\appendix

\section{Technical lemmas}\label{sec:tech.lemmas}

    The Markov property of the GFF, which is a consequence of the strong Markov property of the simple random walk (in the covariance function in \eqref{eqn: G}),
    implies that the value of the field inside a neighborhood is independent of the field outside given the boundary, see e.g.~\citet{MR585179}.
    \hspace{-0.5mm}In particular, for the \hspace{-0.5mm}neighborhood \hspace{-0.5mm}$[v]_\lambda$, where $\lambda\in [0,1]$, this implies
    \begin{equation}\label{eq:GFF.Markov}
        \phi_v(\lambda) \circeq \esp{}{\phi_v \nvert \F_{\partial [v]_{\lambda} \cup [v]_{\lambda}^c}} = \esp{}{\phi_v \nvert \F_{\partial [v]_{\lambda}}}.
    \end{equation}
    Let $v,v'\in V_N$, $\lambda < \lambda'$ and $\mu < \mu'$.
    Another direct consequence is the fact that for $\lambda, \mu > \rho(v,v')$ or $\lambda > \rho(v,v') > \mu'$,
    \begin{equation}\label{eq:GFF.Markov.independence}
        \phi_v(\lambda') - \phi_v(\lambda) \quad \text{is independent of} \quad \phi_{v'}(\mu') - \phi_{v'}(\mu)\ .
    \end{equation}
    This is because the shell $[v]_{\lambda} \cap [v]_{\lambda'}^c$ does not intersect the shell $[v']_{\mu} \cap [v']_{\mu'}^c$ in both cases, see Figure \ref{fig:branching}.
    We stress that, in general, the field $\psi$ does not have the Markov property.
    However, by working with increments of the field $\psi$,  the property analogous to \eqref{eq:GFF.Markov.independence} can be proved.

    \begin{lemma}\label{lem:psi.Markov}
        Let $v,v'\in V_N$, $\lambda < \lambda'$ and $\mu < \mu'$. If we have $\lambda,\mu > \rho(v,v')$ or $\lambda > \rho(v,v') > \mu'$, then
        \begin{equation*}
            \psi_v(\lambda') - \psi_v(\lambda) \quad \text{is independent of} \quad \psi_{v'}(\mu') - \psi_{v'}(\mu)\ .
        \end{equation*}
    \end{lemma}

    \begin{proof}
        Let $v\in V_N$ and $\lambda < \lambda'$.
        By Definition \ref{def:IGFF} of the field $\psi$ and its conditional expectation, we have
        \begin{equation}\label{eq:tech.lemma.1.decomp.prec}
            \psi_v(\lambda) = \hspace{-2mm} \sum_{1 \leq i \leq M} \hspace{-2mm} \sigma_i \hspace{0.5mm}\esp{}{\nabla \phi_v(\lambda_i) \nvert \F_{\partial [v]_{\lambda} \cup [v]_{\lambda}^c}} = \hspace{-2mm} \sum_{\substack{1 \leq i \leq M : \\ \lambda_{i-1} < \lambda}} \hspace{-2mm} \sigma_i \hspace{0.5mm} (\phi_v(\lambda \wedge \lambda_i) - \phi_v(\lambda_{i-1}))\ .
        \end{equation}
        For the last equality, note that, when $\lambda_{i-1} < \lambda$, the increments $\phi_v(\lambda \wedge \lambda_i) - \phi_v(\lambda_{i-1})$ are linear combinations of variables inside the set $\partial [v]_{\lambda} \cup [v]_{\lambda}^c$ and, when $\lambda_i > \lambda$, we have $\mathbb{E}[\phi_v(\lambda_i) - \phi_v(\lambda \vee \lambda_{i-1}) \nvert \F_{\partial [v]_{\lambda} \cup [v]_{\lambda}^c}] = 0$ by the tower property of conditional expectations. By applying the same argument to $\psi_v(\lambda')$, we get
        \begin{equation}\label{eq:tech.lemma.1.decomp}
            \psi_v(\lambda') - \psi_v(\lambda) = \hspace{-11mm} \sum_{\substack{1 \leq i \leq M : \\ \lambda \leq \lambda_{i-1} < \lambda' \text{ or } \lambda < \lambda_i \leq \lambda' \\ \text{or } \lambda_{i-1} \leq \lambda < \lambda' \leq \lambda_i}} \hspace{-11.5mm} \sigma_i \hspace{0.5mm}(\phi_v(\lambda' \wedge \lambda_i) - \phi_v(\lambda \vee \lambda_{i-1}))\ .
        \end{equation}
        The conclusion of the lemma follows from \eqref{eq:GFF.Markov.independence}.
    \end{proof}

    In the remainder of this section, we always assume, without loss of generality, that $N = 2^n$ for some $n\in \N$ and $\lambda n, \lambda'n, \lambda_i n \in \N_0$ for all $i\in \{0,...,M\}$.

    \begin{lemma}\label{lem:tech.lemma.1}
        Let $\delta\in (0,1/2]$ and $\lambda_{i-1} \leq \lambda < \lambda' \leq \lambda_i$ for some $i\in \{1,...,M\}$, then
        \vspace{-3mm}
        \begin{equation}\label{eq:tech.lemma.1}
            -C_1(\delta, \sigma_i) \leq \var{}{\psi_v(\lambda') - \psi_v(\lambda)} - (\lambda' - \lambda) \sigma_i^2 \log N \leq C_2(\sigma_i)
        \end{equation}
        for all $v\in V_N^{\delta}$ and $N$ large enough depending on $\delta$.
        The constant $C_1$ only depends on $\delta$ when $\lambda = 0$.
    \end{lemma}

    \begin{proof}
        The Markov property \eqref{eq:GFF.Markov} yields $\esp{}{\phi_v - \phi_v(\lambda) \nvert \F_{\partial [v]_{\lambda'}}} = \phi_v(\lambda') - \phi_v(\lambda)$.
        Using the conditional variance formula and $\var{}{X \nvert \F} \circeq \esp{}{(X - \esp{}{X \nvert \F})^2 \nvert \F}$, we can compute the variance of \eqref{eq:tech.lemma.1.decomp} in the special case $\lambda_{i-1} \leq \lambda < \lambda' \leq \lambda_i$ :
        \begin{align}\label{eq:tech.lemma.1.green.function.difference.before}
            \var{}{\psi_v(\lambda') - \psi_v(\lambda)}
            &= \sigma_i^2 ~\var{}{\esp{}{\phi_v - \phi_v(\lambda) \nvert \F_{\partial [v]_{\lambda'}}}} \notag \\
            &= \sigma_i^2 \left(\var{}{\phi_v - \phi_v(\lambda)} - \esp{}{\var{}{\phi_v - \phi_v(\lambda) \nvert \F_{\partial [v]_{\lambda'}}}}\right) \notag \\
            &= \sigma_i^2 \left(\var{}{\phi_v - \phi_v(\lambda)} - \var{}{\phi_v - \phi_v(\lambda')}\right).
        \end{align}
        But, it is well known that $\{\phi_u - \esp{}{\phi_u \nvert \F_{\partial B}}\}_{u\in B}$ is a GFF when $B\subseteq \Z^2$ is a finite box, see e.g. \citet{Zeitouni2014ln}.
        Simply choose $B = [v]_s, ~s=\lambda,\lambda'$, in \eqref{eq:tech.lemma.1.green.function.difference.before}, then by the variance definition in \eqref{eqn: G},
        \begin{equation}\label{eq:tech.lemma.1.green.function.difference}
            \var{}{\psi_v(\lambda') - \psi_v(\lambda)} = \sigma_i^2 \hspace{0.5mm}(G_{[v]_{\lambda}}(v,v) - G_{[v]_{\lambda'}}(v,v))\ .
        \end{equation}
        Using standard estimates for the discrete Green function, we can now evaluate the last expression.
        For every finite box $B\subseteq \Z^2$, Proposition 1.6.3 of \citet{MR1117680} shows that (keeping in mind our choice of normalization by $\pi/2$ in \eqref{eqn: G}) :
        \begin{equation}\label{eq:lawler.green.function}
            G_B(x,y) = \left[\sum_{z\in \partial B} \probw{x}{W_{\tau_{\partial B}} = z} a(z - y)\right] - a(y - x), \quad x,y\in B,
        \end{equation}
        where
        \begin{equation}\label{eq:lawler.estime.noyau.potentiel.1}
            a(w) =
            \left\{
            \begin{array}{ll}
                \log(\|w\|_2) + \text{const.} + O(\|w\|_2^{-2}) ~&\mbox{if } w\in \Z^2\backslash\{\boldsymbol{0}\} \\
                0 ~&\mbox{if } w = \boldsymbol{0}
            \end{array}
            \right.
        \end{equation}
        and $\mathscr{P}_x$ is the law of the simple random walk starting at $x\in \Z^2$. Using \eqref{eq:lawler.green.function}, we can rewrite the difference of Green functions in \eqref{eq:tech.lemma.1.green.function.difference} as
        \begin{equation}\label{eq:lawler.green.function.difference}
            \sum_{z\in \partial [v]_{\lambda}} \hspace{-2mm}\probw{v}{W_{\tau_{\partial [v]_{\lambda}}} = z} a(z - v) ~- \hspace{-2mm}\sum_{z\in \partial [v]_{\lambda'}} \hspace{-2mm}\probw{v}{W_{\tau_{\partial [v]_{\lambda'}}} = z} a(z - v)\ .
        \end{equation}
        When $\lambda' = 1$, we have $\|z - v\|_2 = 0$ for $z\in \partial [v]_{\lambda'}$.
        Otherwise, we assumed $v\in V_N^{\delta}$, so take $N$ large enough (depending on $\delta$) that $[v]_{\lambda'}$ is not cut off by $\partial V_N$.
        We have $\|z - v\|_2 \leq \sqrt{2} N^{1-\lambda}$ for $z\in \partial [v]_{\lambda}$ in general and $\|z - v\|_2 \geq \frac{1}{2} N^{1-\lambda'}$ for $z\in \partial [v]_{\lambda'}$ when $\lambda' \neq 1$.
        We deduce the following bound on the variance in \eqref{eq:tech.lemma.1.green.function.difference} :
        \begin{align*}
            \max_{v\in V_N^{\delta}} \var{}{\psi_v(\lambda') - \psi_v(\lambda)}
            &\leq \sigma_i^2 ((1 - \lambda) - (1 - \lambda')) \log N + \sigma_i^2 C \\
            &= (\lambda' - \lambda) \sigma_i^2 \log N + C_2(\sigma_i)\ .
        \end{align*}
        Similarly, we have $\|z - v\|_2 \geq \delta N$ for $z\in \partial [v]_{\lambda}$ when $\lambda = 0$. Otherwise, take $N$ large enough (depending on $\delta$) that $[v]_{\lambda}$ is not cut off by the boundary of $V_N$. We have $\|z - v\|_2 \geq \frac{1}{2} N^{1-\lambda}$ for $z\in \partial [v]_{\lambda}$ when $\lambda \neq 0$ and $\|z - v\|_2 \leq \frac{1}{\sqrt{2}} N^{1-\lambda'}$ for $z\in \partial [v]_{\lambda'}$ in general. We deduce the following bound on the variance in \eqref{eq:tech.lemma.1.green.function.difference} :
        \begin{align*}
            \min_{v\in V_N^{\delta}} \var{}{\psi_v(\lambda') - \psi_v(\lambda)}
            &\geq \sigma_i^2 ((1 - \lambda) - (1 - \lambda')) \log N - \sigma_i^2 C(\delta) \\
            &= (\lambda' - \lambda) \sigma_i^2 \log N - C_1(\delta,\sigma_i)\ .
        \end{align*}
        This ends the proof of the lemma.
    \end{proof}

    Since the upper bound in Lemma \ref{lem:tech.lemma.1} is only valid for $N$ large enough depending on $\delta$, we cannot immediately conclude that it holds for all $v\in V_N$. We show in the next lemma how to extend the bound.

    \begin{lemma}\label{lem:tech.lemma.2}
        Let $\lambda_{i-1} \leq \lambda < \lambda' \leq \lambda_i$ for a certain $i\in \{1,...,M\}$, then
        \begin{equation}\label{eq:tech.lemma.2}
            \max_{v\in V_N} \var{}{\psi_v(\lambda') - \psi_v(\lambda)} \leq (\lambda' - \lambda) \sigma_i^2 \log N + C(\sigma_i)
        \end{equation}
        for $N$ large enough.
    \end{lemma}

    \begin{proof}
        When $v\in \partial V_N$, the bound is trivial because $\psi_v = 0$.
        Therefore, let $v\in V_N^o$.
        To obtain the upper bound on the difference of Green functions in \eqref{eq:tech.lemma.1.green.function.difference}, we only used the fact that \hspace{-0.2mm}$[v]_{\lambda'}$ was not cut off by \hspace{-0.2mm}$\partial V_N$ \hspace{-0.2mm}for \hspace{-0.2mm}$N$ \hspace{-0.2mm}large enough depending on $\delta$. Hence, we only need to show that when $[v]_{\lambda'}$ is cut off, there exists $u\in V_N^o$ such that $[u]_{\lambda'}$ is not cut off and for which
        \begin{equation}\label{eq:lem:tech.lemma.2.diff.green.functions.inequality}
            G_{[v]_{\lambda}}(v,v) - G_{[v]_{\lambda'}}(v,v) \leq G_{[u]_{\lambda}}(u,u) - G_{[u]_{\lambda'}}(u,u) + \tilde{C}(\sigma_i) \ .
        \end{equation}
        Assume that $[v]_{\lambda'}$ is cut off by $\partial V_N$ and choose $u$ to be the center of $V_N$. Clearly, the neighborhood $[u]_{\lambda'}$ is not cut off by the boundary of $V_N$.
        When $\lambda' = 1$, inequality \eqref{eq:lem:tech.lemma.2.diff.green.functions.inequality} is trivial because $G_{[v]_{\lambda'}}(v,v) = G_{[u]_{\lambda'}}(u,u) = 0$ and $G_{[v]_{\lambda}}(v,v) \leq G_{[u]_{\lambda}}(u,u)$ since $[v]_{\lambda}$ is cut off and $[u]_{\lambda}$ is not.
        Now, assume $\lambda' < 1$.
        Denote $\theta(x) \circeq x + u - v$ the translation function that moves $v$ to $u$, see Figure \ref{fig:MAS.u.v.barriere}.
        \vspace{-1mm}
        \begin{figure}[ht]
            \centering
            \includegraphics[scale=0.50]{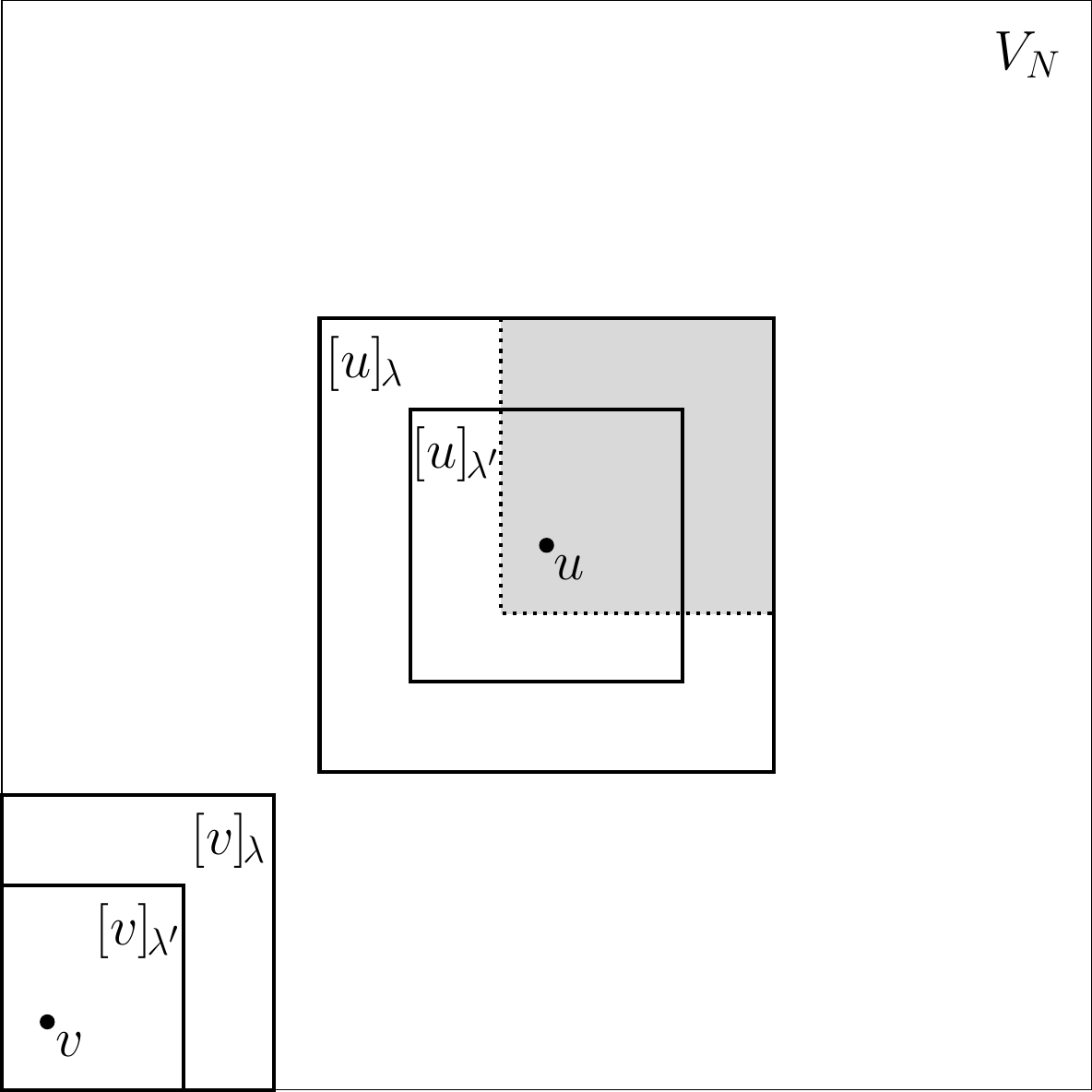}
            \captionsetup{width=0.85\textwidth}
            \caption{The grey area $\theta([v]_{\lambda})$ is the translation of $[v]_{\lambda}$.}
            \label{fig:MAS.u.v.barriere}
        \end{figure}

        \vspace{-1mm}
        For the rest of the proof, redefine $[v]_0$ as the square box of side length $N$ centered at $v$ that has been cut off by $\partial V_N$. Since $\theta([v]_{\lambda}) \subseteq [u]_{\lambda}$, we have
        \begin{align*}
            G_{[v]_{\lambda}}(v,v) - G_{[v]_{\lambda'}}(v,v)
            &= \frac{\pi}{2} \cdot \espw{v}{\sum_{k=\tau_{\partial [v]_{\lambda'}}}^{\tau_{\partial [v]_{\lambda}} - 1} 1_{\{W_k = v\}} 1_{\{\tau_{\partial [v]_{\lambda'}} < \tau_{\partial [v]_{\lambda'} \cap \partial V_N}\}}} \notag \\
            &= \frac{\pi}{2} \cdot \espw{u}{\sum_{k=\tau_{\partial [u]_{\lambda'}}}^{\tau_{\partial \theta([v]_{\lambda})} - 1} 1_{\{W_k = u\}} 1_{\{\tau_{\partial [u]_{\lambda'}} < \tau_{\theta(\partial [v]_{\lambda'} \cap \partial V_N)}\}}} \\
            &\leq \frac{\pi}{2} \cdot \espw{u}{\sum_{k=\tau_{\partial [u]_{\lambda'}}}^{\tau_{\partial [u]_{\lambda}} - 1} 1_{\{W_k = u\}}} = G_{[u]_{\lambda}}(u,u) - G_{[u]_{\lambda'}}(u,u)\ .
        \end{align*}
        This proves \eqref{eq:lem:tech.lemma.2.diff.green.functions.inequality} when $\lambda \neq 0$. Since $[v]_0 \circeq V_N$ throughout the article and we defined $[v]_0$ differently in this proof, it remains to show that
        \begin{equation}\label{eq:lem:tech.lemma.2.max.var.lambda.0}
            \max_{v\in V_N} G_{V_N}(v,v) - G_{[v]_0}(v,v) \leq \tilde{C}(\sigma_i)
        \end{equation}
        for \eqref{eq:lem:tech.lemma.2.diff.green.functions.inequality} to be true when $\lambda = 0$.
        By the strong Markov property and \eqref{eq:lawler.green.function} :
        \begin{align*}
            &G_{V_N}(v,v) - G_{[v]_0}(v,v)
            = \hspace{-3mm} \sum_{z\in \partial [v]_0 \cap V_N^o} \hspace{-4mm} \probw{v}{W_{\tau_{\partial [v]_0}} = z} G_{V_N}(z,v) \\
            &\quad= \hspace{-3mm} \sum_{z\in \partial [v]_0 \cap V_N^o} \hspace{-4mm} \probw{v}{W_{\tau_{\partial [v]_0}} = z} \hspace{-1mm}\sum_{w\in \partial V_N} \hspace{-1mm} \probw{z}{W_{\tau_{\partial V_N}} = w} (a(w - v) - a(v - z))\ .
        \end{align*}
        But $\|w - v\|_2 \leq \sqrt{2} N$ for all $w\in \partial V_N$ and $\|v - z\|_2 \geq \frac{1}{2} N$ for all $z\in \partial [v]_0 \cap V_N^o$. We get the desired conclusion using \eqref{eq:lawler.estime.noyau.potentiel.1}.
    \end{proof}

    In \hspace{-0.2mm}order \hspace{-0.2mm}to \hspace{-0.2mm}approximate \hspace{-0.2mm}the \hspace{-0.2mm}branching \hspace{-0.2mm}structure \hspace{-0.2mm}of \hspace{-0.2mm}the \hspace{-0.2mm}$(\boldsymbol{\sigma},\boldsymbol{\lambda})$-GFF \hspace{-0.2mm}in \hspace{-0.2mm}Lemma \hspace{-0.2mm}\ref{lem:IGFF.order1.upper.bound} and Lemma \ref{lem:IGFF.high.points.upper.bound}, we need to show that the variance of $\psi_v(\lambda) - \psi_{v_{\lambda}}(\lambda)$ is bounded by a constant, where $v_{\lambda}$ denotes any representative in $R_{\lambda}$ that is closest to $v$. Our final goal here is to show Lemma \ref{lem:tech.lemma.3}. We start by proving a more general version of Lemma 12 found in \citet{MR1880237}. We define
    \begin{equation*}
        \phi_v(A) \circeq \esp{}{\phi_v \nvert \F_{\partial (A \cap V_N)}}
    \end{equation*}
    and $d(z,A) \circeq \min_{w\in A} \|z - w\|_2$ for any non-empty set $A\subseteq \Z^2$.

    \begin{lemma}\label{lem:CLG.lemme12.BDG2001}
        Let $B\subseteq \Z^2$ be a square box of width smaller or equal to $N/2$ such that $B \cap V_N \neq \emptyset$. Moreover, let $0 \leq \eta < 1$ and $L\in \{1,2,...,N/4\}$, then there exists a constant $C = C(\eta) > 0$ such that
        \begin{equation}\label{eq:lem:CLG.lemme12.BDG2001.bound.lemma}
            \max_{\substack{u,v\in B \cap V_N \\ d(u,\partial B) = L \\ \|u - v\|_2 \leq \eta L}} \hspace{-3mm}\var{}{\phi_u(B) - \phi_v(B)} \leq C\ .
        \end{equation}
    \end{lemma}

    \begin{proof}
        Let $u,v\in B \cap V_N$ be such that $d(u,\partial B) = L$ and $\|u - v\|_2 \leq \eta L$. Denote $\overbar{B} \circeq B \cap V_N$. Using the conditional variance formula as in \eqref{eq:tech.lemma.1.green.function.difference.before}, we have
        \begin{align}\label{eq:lem:CLG.lemme12.BDG2001.equation.1}
            \var{}{\phi_u(B) - \phi_v(B)}
            &= \var{}{\esp{}{\phi_u - \phi_v \nvert \mathcal{F}_{\partial \overbar{B}}}} \notag \\
            &= \var{}{\phi_u - \phi_v} - \esp{}{\var{}{\phi_u - \phi_v \nvert \mathcal{F}_{\partial \overbar{B}}}} \notag \\
            &= (G_{V_N}(u,u) - G_{V_N}(u,v) + G_{V_N}(v,v) - G_{V_N}(v,u)) \notag \\
            &\quad - (G_{\overbar{B}}(u,u) - G_{\overbar{B}}(u,v) + G_{\overbar{B}}(v,v) - G_{\overbar{B}}(v,u))\ .
        \end{align}
        For this proof, redefine $[u]_0$ as the square box of side length $N$ centered at $u$ that has been cut off by $\partial V_N$. From \eqref{eq:lem:tech.lemma.2.max.var.lambda.0}, we know $\max_{u\in V_N} G_{V_N}(u,u) - G_{[u]_0}(u,u) \leq C$. Using the exact same method, we can also easily show that
        \begin{equation*}
            \max_{\substack{v\in V_N \\ \|u - v\|_2 \leq \eta N/2}} \hspace{-4mm}G_{V_N}(v,v) - G_{[u]_0}(v,v) \leq C(\eta)
        \end{equation*}
        because we would have $\|v - z\|_2 \geq (1 - \eta) N/2$ for all $z\in \partial [u]_0 \cap V_N^o$ in the reasoning below \eqref{eq:lem:tech.lemma.2.max.var.lambda.0}, where $\eta < 1$ by hypothesis. Finally, $-G_{V_N}(u,v) \leq -G_{[u]_0}(u,v)$, so proving \eqref{eq:lem:CLG.lemme12.BDG2001.bound.lemma} boils down to the proof of the following inequality :
        \vspace{0.5mm}
        \begin{equation}\label{eq:lem:CLG.lemme12.BDG2001.to.show}
            (\clubsuit) \circeq
            \left\{\hspace{-1mm}
            \begin{array}{ll}
                (G_{[u]_0}(u,u) - G_{[u]_0}(u,v)) - (G_{\overbar{B}}(u,u) - G_{\overbar{B}}(u,v)) \\
                + \hspace{0.5mm}(G_{[u]_0}(v,v) - G_{[u]_0}(v,u)) - (G_{\overbar{B}}(v,v) - G_{\overbar{B}}(v,u))
            \end{array}\hspace{-1.5mm}
            \right\} \leq \tilde{C}(\eta)\hspace{0.5mm}.\vspace{1mm}
        \end{equation}
        To show \eqref{eq:lem:CLG.lemme12.BDG2001.to.show}, we consider two cases : $d(u,\partial V_N) \leq L$ and $d(u,\partial V_N) > L$.
        \vspace{1mm}
        \begin{equation*}
            \boxed{\text{Case 1 : } d(u,\partial V_N) \leq L}
        \end{equation*}
        \vspace{-2mm}

        Since $\overbar{B} \subseteq [u]_0$ (recall that $B$ is a square box of width smaller or equal to $N/2$ and contains $u$), then we always have
        \begin{equation}\label{eq:lem:CLG.lemme12.BDG2001.new.1}
            (\clubsuit) \leq (G_{[u]_0}(u,u) - G_{\overbar{B}}(u,u)) + (G_{[u]_0}(v,v) - G_{\overbar{B}}(v,v))\ .
        \end{equation}
        Note that the box $B$ is cut off by $\partial V_N$ in Case 1.
        By translating $u,v,B$ together in such a way that $u$ doesn't get closer to $\partial V_N$ with respect to both axes, each difference of Green functions in \eqref{eq:lem:CLG.lemme12.BDG2001.new.1} can only increase (see the argument below Figure \ref{fig:MAS.u.v.barriere}).
        Therefore, it is sufficient to bound \eqref{eq:lem:CLG.lemme12.BDG2001.new.1} when $d(u,\partial V_N) = L$.
        Assume $d(u,\partial V_N) = L$ for the rest of Case $1$.
        Since $d(u,\partial B) = L$ by hypothesis, we have $d(u,\partial \overbar{B}) = L$ and we get $d(v,\partial \overbar{B}) \geq \lceil (1 - \eta) L \rceil \geq 1$ by the triangle inequality.
        Consequently,
        \begin{equation}\label{eq:lem:CLG.lemme12.BDG2001.new.2}
            (\clubsuit) \leq G_{[u]_0}(u,u) + G_{[u]_0}(v,v) - 2 \log L + C(\eta)
        \end{equation}
        using \eqref{eq:lawler.green.function} and \eqref{eq:lawler.estime.noyau.potentiel.1}.

        By the symmetries of the square, we can assume, without loss of generality, that the minimum in $d(u,\partial V_N) = L$ is achieved on the bottom edge of $V_N$ (which lies on the $x$-axis). Define the half-space $\mathcal{H} \circeq \{z = (z_1,z_2)\in \Z^2 \nvert z_2 \geq 0\}$. Since we have $[u]_0 \subseteq \mathcal{H}$ and $d(v,\partial \mathcal{H}) \leq (1 + \eta) L$, by the triangle inequality, then
        \begin{equation}\label{eq:lem:CLG.lemme12.BDG2001.new.3}
            (\clubsuit) \leq 2 \hspace{-13.5mm} \max_{\substack{z\in \mathcal{H} \\ \lceil (1 - \eta) L \rceil \leq d(z,\partial \mathcal{H}) \leq (1+ \eta) L}} \hspace{-13.5mm} G_{\mathcal{H}}(z,z) - 2 \log L + C(\eta)\ .
        \end{equation}
        From Proposition 8.1.1 of \cite{MR2677157},
        \begin{equation}\label{eq:lem:CLG.lemme12.BDG2001.new.4}
            G_{\mathcal{H}}(z,z) = a(z - \bar{z}) \stackrel{\eqref{eq:lawler.estime.noyau.potentiel.1}}{=} \log(\|z - \bar{z}\|_2) + \text{const.} + O(\|z - \bar{z}\|_2^{-2})
        \end{equation}
        where $z = (z_1,z_2)$ and $\bar{z} \circeq (z_1,-z_2)$. The conclusion for Case 1 follows from \eqref{eq:lem:CLG.lemme12.BDG2001.new.3} because $2 \leq 2 \lceil (1 - \eta) L \rceil \leq \|z - \bar{z}\|_2 \leq 2 (1 + \eta) L$ in \eqref{eq:lem:CLG.lemme12.BDG2001.new.4}.

        \begin{equation*}
            \boxed{\text{Case 2 : } d(u,\partial V_N) > L}
        \end{equation*}
        \vspace{-1mm}

        For Case $2$, we follow the argument from \citet{MR1880237}. We give the details for convenience.
        For all $k\in \N_0$, define $[u]_0^k \subseteq \Z^2$ the square box of side length $2^k N$ centered at $u$ (not cut off by anything). For instance, $[u]_0 = [u]_0^0 \cap V_N$ in this proof. Note that $[u]_0 \subseteq [u]_0^1 \subseteq [u]_0^2 \subseteq ...$ and $[u]_0 \cup \bigcup_{k=1}^{\infty} [u]_0^k = \Z^2$, so
        \begin{align*}
            (\clubsuit)
            &\leq \left\{\hspace{-1mm}
            \begin{array}{ll}
                (G_{[u]_0}(u,u) - G_{[u]_0}(u,v)) - (G_{\overbar{B}}(u,u) - G_{\overbar{B}}(u,v)) \\
                + \hspace{0.5mm}(G_{[u]_0}(v,v) - G_{[u]_0}(v,u)) - (G_{\overbar{B}}(v,v) - G_{\overbar{B}}(v,u))
            \end{array}\hspace{-1.5mm}
            \right\} \notag \\
            &\quad+
            \sum_{k=1}^{\infty} \hspace{0.3mm}
            \left\{\hspace{-1mm}
            \begin{array}{ll}
                (G_{[u]_0^k}(u,u) - G_{[u]_0^k}(u,v)) - (G_{[u]_0^{k-1}}(u,u) - G_{[u]_0^{k-1}}(u,v)) \\
                + \hspace{0.5mm}(G_{[u]_0^k}(v,v) - G_{[u]_0^k}(v,u)) - (G_{[u]_0^{k-1}}(v,v) - G_{[u]_0^{k-1}}(v,u))
            \end{array}\hspace{-1.5mm}
            \right\} \\
            &= \frac{\pi}{2} \cdot \espw{u}{\sum_{k=\tau_{\partial \overbar{B}}}^{\infty} (1_{\{W_k = u\}} - 1_{\{W_k = v\}})} + \frac{\pi}{2} \cdot \espw{v}{\sum_{k=\tau_{\partial \overbar{B}}}^{\infty} (1_{\{W_k = v\}} - 1_{\{W_k = u\}})} \hspace{-1mm}.
        \end{align*}
        The inequality comes from the fact that each pair of braces in the infinite sum is equal to $\mathbb{V}_{[u]_0^k}(\mathbb{E}[\phi_u - \phi_v \nvert \F_{\partial [u]_0^{k-1}}]) \geq 0$ by steps analogous to \eqref{eq:lem:CLG.lemme12.BDG2001.equation.1}. The equality follows because the infinite sum is telescopic.

        By conditioning on the point $z\in \partial \overbar{B}$ where the simple random walk starting at $u$ or $v$ will be when hitting the boundary of $\overbar{B}$, and using the strong Markov property, we deduce
        \begin{align}\label{eq:lem:CLG.lemme12.BDG2001.clubsuit.last}
             (\clubsuit)
             &\leq \hspace{-1mm}\sum_{z\in \partial \overbar{B}} \hspace{-1mm} \left(\mathscr{P}_{u}(W_{\tau_{\partial \overbar{B}}} \hspace{-1mm} = \hspace{-0.5mm}z) \hspace{-0.2mm}- \hspace{-0.2mm}\mathscr{P}_{v}(W_{\tau_{\partial \overbar{B}}} \hspace{-1mm} = \hspace{-0.5mm}z)\right) \cdot \frac{\pi}{2} \cdot \mathscr{E}_z \hspace{-1mm}\left[\sum_{k=0}^{\infty} (1_{\{W_k = u\}} - 1_{\{W_k = v\}})\right] \notag \\
             &= \hspace{-1mm}\sum_{z\in \partial \overbar{B}} \hspace{-1mm} \left(\mathscr{P}_{u}(W_{\tau_{\partial \overbar{B}}} \hspace{-1mm} = \hspace{-0.5mm}z) \hspace{-0.2mm}- \hspace{-0.2mm}\mathscr{P}_{v}(W_{\tau_{\partial \overbar{B}}} \hspace{-1mm} = \hspace{-0.5mm}z)\right) \cdot (a(v - z) - a(u - z))
        \end{align}
        where `` $a$ '', the \textit{potential kernel} (see p.37 in \citet{MR1117680}), is defined by
        \begin{equation*}
            a(w) \circeq \frac{\pi}{2} \cdot \espw{\boldsymbol{0}}{\sum_{k=0}^{\infty} (1_{\{W_k = \boldsymbol{0}\}} - 1_{\{W_k = w\}})}.
        \end{equation*}
        Theorem 1.6.2 in \citet{MR1117680} shows that this is the same function as in \eqref{eq:lawler.estime.noyau.potentiel.1}. Therefore, we can evaluate \eqref{eq:lem:CLG.lemme12.BDG2001.clubsuit.last} :
        \begin{equation}\label{eq:lem:CLG.lemme12.BDG2001.equation.3}
            a(v - z) - a(u - z) = \log\left(\frac{\|v - z\|_2}{\|u - z\|_2}\right) + O(\|v - z\|_2^{-2}) - O(\|u - z\|_2^{-2})\ .
        \end{equation}
        By the triangle inequality, we have
        \begin{equation}\label{eq:lem:CLG.lemme12.BDG2001.equation.4}
            \log\left(1 - \frac{\|u - v\|_2}{\|u - z\|_2}\right) \leq \log\left(\frac{\|v - z\|_2}{\|u - z\|_2}\right) \leq \log\left(1 + \frac{\|u - v\|_2}{\|u - z\|_2}\right)\ .
        \end{equation}
        Now, notice that
        \begin{itemize}
            \item $\|u - v\|_2 \leq \eta L$ by hypothesis ;
            \item $\|u - z\|_2 \geq L$ for all $z\in \partial \overbar{B}$ by the assumption of Case $2$ ;
            \item $\|v - z\|_2 \geq \|u - z\|_2 - \|u - v\|_2 \geq \lceil (1 - \eta) L \rceil$ for all $z\in \partial \overbar{B}$, from the first two bullets and the triangle inequality.
        \end{itemize}
        Using the three bullets in \eqref{eq:lem:CLG.lemme12.BDG2001.equation.3} and \eqref{eq:lem:CLG.lemme12.BDG2001.equation.4}, we have
        \begin{equation}\label{eq:lem:CLG.lemme12.BDG2001.equation.5}
            \log(1 - \eta) - \frac{C_1}{\lceil (1 - \eta) L \rceil^2} \leq \eqref{eq:lem:CLG.lemme12.BDG2001.equation.3} \leq \log(1 + \eta) + \frac{C_2}{\lceil (1 - \eta) L \rceil^2}
        \end{equation}
        for appropriate constants $C_1,C_2 > 0$.
        Since $L \geq 1$ and $\lceil (1 - \eta) L \rceil \geq 1$, inequality \eqref{eq:lem:CLG.lemme12.BDG2001.to.show} follows by regrouping \eqref{eq:lem:CLG.lemme12.BDG2001.clubsuit.last}, \eqref{eq:lem:CLG.lemme12.BDG2001.equation.3} and \eqref{eq:lem:CLG.lemme12.BDG2001.equation.5}.
    \end{proof}

    \begin{lemma}\label{lem:CLG.lemme12.BDG2001.complement}
        Let $0 \leq \lambda' < 1$ and $d \geq 1/\sqrt{2}$. For all $v\in V_N$, define $S_{v,d}$ to be the set of finite boxes $B \subseteq \Z^2$ such that $[v]_{\lambda'} \subseteq B \cap V_N$ and $\max_{z\in \partial B} \|v - z\|_2 \leq d N^{1- \lambda'}$, then there exists a constant $C = C(d) > 0$ such that
        \begin{equation*}
            \max_{v\in V_N} \max_{B\in S_{v,d}} \var{}{\phi_v(\lambda') - \phi_v(B)} \leq C
        \end{equation*}
        for $N$ large enough.
    \end{lemma}

    \begin{proof}
        This follows directly from the calculations in Lemma \ref{lem:tech.lemma.1} and Lemma \ref{lem:tech.lemma.2} where $B \cap V_N$ plays the same role as $[v]_{\lambda}$.
    \end{proof}

    The next lemma is used in equation \eqref{eqn: H split.1} of Lemma \ref{lem:IGFF.order1.upper.bound} and equation \eqref{eq:lem:IGFF.high.points.upper.bound.beginning.2} of Lemma \ref{lem:IGFF.high.points.upper.bound} to show that the error coming from the approximation of the branching structure of $\psi$ is small enough that the problem of finding the upper bound for the maximum and the log-number of $\gamma$-high points is the same (modulo the additional hurdle caused by the decay of variance near the edges of $V_N$) as in the context of branching random walks.

    \begin{lemma}\label{lem:tech.lemma.3}
        Let $\lambda_{j-1} < \lambda \leq \lambda_j$ for a certain $j\in \{1,...,M\}$, then there exists a constant $C = C(\sigma_1,...,\sigma_j) > 0$ such that
        \begin{equation*}
            \max_{v\in V_N} \var{}{\psi_v(\lambda) - \psi_{v_{\lambda}}(\lambda)} \leq C
        \end{equation*}
        for $N$ large enough.
    \end{lemma}

    \begin{proof}
        The lemma is trivial when $\lambda = 1$ since $v = v_1$. Therefore, assume $0 < \lambda < 1$. Choose \hspace{-0.1mm}$v_{\lambda}\hspace{-0.4mm}\in \hspace{-0.4mm}R_{\lambda}$ \hspace{-0.1mm}any \hspace{-0.1mm}representative \hspace{-0.1mm}that \hspace{-0.1mm}is \hspace{-0.1mm}closest \hspace{-0.1mm}to \hspace{-0.1mm}$v$ \hspace{-0.1mm}(there \hspace{-0.1mm}may \hspace{-0.1mm}be \hspace{-0.1mm}more \hspace{-0.1mm}than \hspace{-0.1mm}one). For all $\mu \in (0,\lambda]$, the square box $B_{\mu} \subseteq \Z^2$ of width $2 \lceil N^{1 - \mu} \rceil$ centered at $v_{\lambda}$ contains both $[v]_{\mu}$ and $[v_{\lambda}]_{\mu}$ because $\|v - v_{\lambda}\|_{\infty} \leq \frac{1}{2} N^{1 - \lambda}$.
        Then, by Jensen's inequality :
        \vspace{1mm}
        \begin{equation}\label{lem:tech.lemma.3.star}
            \var{}{\phi_v(\mu) - \phi_{v_{\lambda}}(\mu)}
            \leq
            3 \cdot \left\{\hspace{-1mm}
            \begin{array}{l}
                \hspace{4mm}\var{}{\phi_v(\mu) - \phi_v(B_{\mu})} \\
                + ~\var{}{\phi_v(B_{\mu}) - \phi_{v_{\lambda}}(B_{\mu})} \\
                + ~\var{}{\phi_{v_{\lambda}}(B_{\mu}) - \phi_{v_{\lambda}}(\mu)}
            \end{array}\hspace{-1mm}
            \right\} \leq \tilde{C}\ .\vspace{1mm}
        \end{equation}
        To see the last inequality, bound the first and third variance term inside the braces using Lemma \ref{lem:CLG.lemme12.BDG2001.complement} with $d = 3/\sqrt{2}$ and bound the second variance term inside the braces using Lemma \ref{lem:CLG.lemme12.BDG2001} with $\eta = 1/\sqrt{2}$ (since $\|v - v_{\lambda}\|_2 \leq N^{1 - \lambda} / \sqrt{2} \leq N^{1 - \mu} / \sqrt{2}$), $u = v_{\lambda}$ and $L = \lceil N^{1-\mu} \rceil$. Now, from \eqref{eq:tech.lemma.1.decomp.prec} and Jensen's inequality, we get
        \vspace{0.5mm}
        \begin{align*}
            \var{}{\psi_v(\lambda) - \psi_{v_{\lambda}}(\lambda)}
            &= \mathbb{V}\left(\hspace{-1mm}
            \begin{array}{l}
                \hspace{3.5mm}\sigma_j (\phi_v(\lambda) - \phi_{v_{\lambda}}(\lambda)) \\
                + \sum_{i=1}^{j-1} (\sigma_i - \sigma_{i+1}) (\phi_v(\lambda_i) - \phi_{v_{\lambda}}(\lambda_i))
            \end{array}
            \hspace{-1mm}\right) \\
            &\leq j \cdot \left\{\hspace{-1mm}
            \begin{array}{l}
                \hspace{3.5mm}\sigma_j^2 ~\var{}{\phi_v(\lambda) - \phi_{v_{\lambda}}(\lambda)} \\
                + \sum_{i=1}^{j-1} (\sigma_i - \sigma_{i+1})^2 ~\var{}{\phi_v(\lambda_i) - \phi_{v_{\lambda}}(\lambda_i)}
            \end{array}\hspace{-1mm}
            \right\}.
        \end{align*}
        Simply use \eqref{lem:tech.lemma.3.star} to bound each variance term inside the braces by a constant. This ends the proof of the lemma.
    \end{proof}

    \begin{lemma}[Gaussian estimates, see e.g. \citet{MR2319516}]\label{lem:gaussian.estimates}
        Suppose that $Z\sim \mathcal{N}(0,\sigma^2)$ where $\sigma > 0$, then for all $z > 0$,
        \begin{equation*}
             \left(1 - \frac{\sigma^2}{z^2}\right) \frac{\sigma}{\sqrt{2\pi} z} \exp\left(-\frac{z^2}{2 \sigma^2}\right) \leq \prob{}{Z \geq z} \leq \frac{\sigma}{\sqrt{2\pi} z} \exp\left(-\frac{z^2}{2 \sigma^2}\right).
        \end{equation*}
    \end{lemma}

    \vspace{0.3mm}
    \begin{lemma}[\citet{PaleyZygmund1932} inequality]\label{lem:paley.zygmund}
        Let $0 \leq X \in L^2(\mathbb{P})$ be such that $\prob{}{X > 0} > 0$, then for all $0 \leq \theta \leq 1$,
        \vspace{-1mm}
        \begin{equation*}
            \prob{}{X \geq \theta \esp{}{X}} \geq (1 - \theta)^2 \frac{(\esp{}{X})^2}{\esp{}{X^2}}\ .
        \end{equation*}
    \end{lemma}

\section{Karush-Kuhn-Tucker theorem and applications}\label{sec:KKT}

    In this section, we state the Karush-Kuhn-Tucker theorem and the solutions to the two optimization problems posed in Section \ref{sec:outline.proofs}. The optimal path for the maximum, $\lambda \mapsto L_N^{\star}(\lambda)$, comes from the solution to the problem stated in Lemma \ref{lem:optimization.1} while the optimal path for $\gamma$-high points, $\lambda \mapsto L_N^{\gamma}(\lambda)$, comes from the solution to the problem stated in Lemma \ref{lem:optimization.2}. The Karush-Kuhn-Tucker theorem only gives, a priori, necessary conditions for local optimality. However, the conditions are also sufficient for global optimality here because the objective function $f_{\gamma}$ below and the constraint functions $g_k$ are continuously differentiable and concave ($f_{\gamma^{\star}}$ is linear), see \citet{MR614849}. The proof of the two lemmas can be found in Appendix A of \citet{Ouimet2014master} and are direct applications of the theorem.

    \begin{theorem}[Karush-Kuhn-Tucker, see e.g. \hspace{-0.5mm}\citet{MR2954022}]\label{thm:KKT}
        Let $f : \R^{n_1} \rightarrow \R$ be an objective function and let
        \begin{equation*}
            \mathcal{U}^{\geq} \circeq \{\boldsymbol{x}\in \R^{n_1} \nvert g_k(\boldsymbol{x}) \geq 0 \ \ \forall k\in \{1,...,n_2\}\}
        \end{equation*}
        be a set of constraints specified by the constraint functions \hspace{0.4mm}$g_k \hspace{-0.6mm}: \hspace{-0.2mm}\R^{n_1} \hspace{-0.4mm}\rightarrow \hspace{-0.2mm}\R, ~\hspace{-0.2mm}1 \hspace{-0.2mm}\leq \hspace{-0.2mm}k \hspace{-0.2mm}\leq \hspace{-0.2mm}n_2$. Furthermore, assume that
        \begin{itemize}
            \item[\ (a)] $f$ attains a local maximum at $\boldsymbol{x}^{\star}\in \mathcal{U}^{\geq}$ with respect to $\mathcal{U}^{\geq}$;
            \item[\ (b)] $f$ is Fr\'echet differentiable at $\boldsymbol{x}^{\star}$;
            \item[\ (c)] the $g_k$'s are Fr\'echet differentiable at $\boldsymbol{x}^{\star}$.
        \end{itemize}
        When the constraints qualify (they do in Lemma \ref{lem:optimization.1} and Lemma \ref{lem:optimization.2} because the $g_k$'s are concave and $\boldsymbol{0}\in \mathcal{U}^{>}$, see Slater's condition in \citet{MR2954022}), then there exists $(\mu_1,...,\mu_{n_2})\in \R^{n_2}$ such that the following points hold for all $k\in \{1,...,n_2\}$ ($\nabla$ is the gradient here) :
        \begin{itemize}
            \item[\ (1)] $\nabla f(\boldsymbol{x}^{\star}) + \sum_{k=1}^{n_2} \mu_k \nabla g_k(\boldsymbol{x}^{\star}) = 0$;
            \item[\ (2)] $g_k(\boldsymbol{x}^{\star}) \geq 0$;
            \item[\ (3)] $\mu_k \geq 0$;
            \item[\ (4)] $\mu_k g_k(\boldsymbol{x}^{\star}) = 0$.
        \end{itemize}
    \end{theorem}

    \begin{lemma}(Optimal path for the maximum)\label{lem:optimization.1}
        Let
        \begin{equation*}
            f_{\gamma^{\star}}(x_1,...,x_M) \circeq \sum_{i=1}^M x_i
        \end{equation*}
        be the objective function to maximize under the constraints
        \begin{equation*}
            g_k(x_1,...,x_M) \circeq \sum_{i=1}^k \left(\nabla \lambda_i - \frac{x_i^2}{\sigma_i^2 \nabla \lambda_i}\right) \geq 0, \ \ \ 1 \leq k \leq M,
        \end{equation*}
        then there exists a unique global maximum. The solution is given by
        \begin{equation*}
            x_i^{\star} = \nabla \mathcal{J}_{\sigma^2 / \bar{\sigma}}(\lambda_i), \ \ \ 1 \leq i \leq M,
        \end{equation*}
        and the maximum is given by
        \begin{equation*}
             f_{\gamma^{\star}}(x_1^{\star},...,x_M^{\star}) = \mathcal{J}_{\sigma^2 / \bar{\sigma}}(1) \circeq \gamma^{\star}.
        \end{equation*}
    \end{lemma}

    \begin{lemma}(Optimal path for $\gamma$-high points)\label{lem:optimization.2}
        Let $\gamma^{l-1} \leq \gamma < \gamma^l$ for a certain $l\in \{1,...,m\}$, where the critical levels $\gamma^l$ are defined in \eqref{eq:critic.levels}. Furthermore, let
        \begin{equation*}
            f_{\gamma}(x_1,...,x_{M-1}) \circeq \sum_{i=1}^{M-1} \left(\nabla \lambda_i - \frac{x_i^2}{\sigma_i^2 \nabla \lambda_i}\right) + \left(\nabla \lambda_M - \frac{(\gamma - \sum_{i'=1}^{M-1} x_{i'})^2}{\sigma_M^2 \nabla \lambda_M}\right)
        \end{equation*}
        be the objective function to maximize under the constraints
        \begin{equation*}
            g_k(x_1,...,x_{M-1}) \circeq \sum_{i=1}^k \left(\nabla \lambda_i - \frac{x_i^2}{\sigma_i^2 \nabla \lambda_i}\right) \geq 0, \ \ \ 1 \leq k \leq M-1\hspace{0.3mm},
        \end{equation*}
        then there exists a unique global maximum. The solution is given by
        \begin{equation*}
            x_i^{\star} =
            \left\{
            \begin{array}{ll}
                \nabla \mathcal{J}_{\sigma^2 / \bar{\sigma}}(\lambda_i) ~& \mbox{when } \lambda_i \leq \lambda^{l-1} \\
                \frac{\nabla \mathcal{J}_{\sigma^2}(\lambda_i)}{\mathcal{J}_{\sigma^2}(\lambda^{l-1},1)} (\gamma - \mathcal{J}_{\sigma^2 / \bar{\sigma}}(\lambda^{l-1})) ~& \mbox{when } \lambda_i > \lambda^{l-1}
            \end{array}
            \right.
        \end{equation*}
        for all $i\in \{1,...,M-1\}$ and the maximum is given by
        \begin{equation*}
          f_{\gamma}(x_1^{\star},...,x_{M-1}^{\star}) = (1 - \lambda^{l-1}) - \frac{\left(\gamma - \mathcal{J}_{\sigma^2 / \bar{\sigma}}(\lambda^{l-1})\right)^2}{\mathcal{J}_{\sigma^2}(\lambda^{l-1},1)} \circeq \mathcal{E}_{\gamma}\ .
        \end{equation*}
    \end{lemma}

\section*{Acknowledgements}

    The authors would like to thank an anonymous referee for a careful review of the preliminary version. We also gratefully acknowledge insightful discussions with David Belius, Anton Bovier, Nicola Kistler and Olivier Zindy.

%
% ----------  B I B L I O G R A P H I E  --------------------------------------------------------------------------------------------------------------------------------------
%

\bibliographystyle{alea2}
%\bibliography{Arguin_Ouimet_2016_bib}

\end{document}